\documentclass{amsart}
\usepackage{amsmath}
\usepackage{amssymb}
\usepackage{amscd}
\usepackage{amsfonts}
\usepackage{latexsym}
\usepackage{amsfonts}
\usepackage{graphicx}
\usepackage[all]{xy}
\usepackage{mathrsfs}
\usepackage{amscd}
\usepackage{verbatim}

\newcommand{\leftlim}{\mathop{\displaystyle\lim_{\longleftarrow}}}
\newcommand{\scrH}{{\mathscr{H}}}

\numberwithin{equation}{section}

\newcommand{\E}{\mathcal{E}}
\newcommand{\F}{\mathcal{F}}

\newcommand{\cO}{\mathcal{O}}

\newcommand{\sD}{{\mathscr D}}

\newcommand{\dbQ}{{\mathbb{Q}}}

\newcommand{\BA}{{\mathbb{A}}}

\newcommand{\BD}{{\mathbb{D}}}
\newcommand{\BG}{{\mathbb{G}}}
\newcommand{\BF}{{\mathbb{F}}}
\newcommand{\BN}{{\mathbb{N}}}
\newcommand{\BP}{{\mathbb{P}}}
\newcommand{\BQ}{{\mathbb{Q}}}

\newcommand{\BZ}{{\mathbb{Z}}}

\newcommand{\fp}{{\mathfrak p}}
\newcommand{\fm}{{\mathfrak m}}

\newcommand{\fT}{{\mathfrak T}}

\newcommand{\dbQbar}{\overline{\dbQ}}

\newcommand{\BQbar}{\overline{\dbQ}}

\newcommand{\iso}{\buildrel{\sim}\over{\longrightarrow}}
\newcommand{\epi}{\twoheadrightarrow}
\newcommand{\mono}{\hookrightarrow}

\DeclareMathOperator{\Br}{{Br}}
\DeclareMathOperator{\Card}{{Card}}
\DeclareMathOperator{\Tr}{{Tr}}
\DeclareMathOperator{\Sh}{{Sh}}
\DeclareMathOperator{\Spec}{{Spec}}
\DeclareMathOperator{\Gal}{{Gal}}
\DeclareMathOperator{\Hilb}{{Hilb}}
\DeclareMathOperator{\red}{{red}}
\DeclareMathOperator{\codim}{{codim}}
\DeclareMathOperator{\tame}{{tame}}
\DeclareMathOperator{\mix}{{mix}}
\DeclareMathOperator{\mot}{{mot}}
\DeclareMathOperator{\LS}{{LS}}
\DeclareMathOperator{\LLS}{{\mathcal{LS}}}
\DeclareMathOperator{\Hom}{{Hom}}
\newcommand{\HOM}{\underline{\operatorname{Hom}}}
\DeclareMathOperator{\Mor}{{Mor}}
\DeclareMathOperator{\End}{{End}}
\DeclareMathOperator{\Ext}{{Ext}}
\DeclareMathOperator{\Aut}{{Aut}}
\DeclareMathOperator{\Ker}{{Ker}}
\DeclareMathOperator{\Fr}{{Fr}}
\DeclareMathOperator{\im}{{Im}}
\DeclareMathOperator{\Perv}{{Perv}}
\DeclareMathOperator{\Pol}{{Pol}}
\DeclareMathOperator{\Res}{{Res}}

\theoremstyle{plain}
\newtheorem{theorem}{Theorem}[section]
\newtheorem{lem}[theorem]{Lemma}
\newtheorem{prop}[theorem]{Proposition}
\newtheorem{defin}[theorem]{Definition}
\newtheorem{cor}[theorem]{Corollary}

\newtheorem{conj}[theorem]{Conjecture}

\theoremstyle{definition}
\newtheorem{rems}[theorem]{Remarks}
\newtheorem{rem}[theorem]{Remark}
\newtheorem{example}[theorem]{Example}

\newtheorem{quest}[theorem]{Question}

\newcommand{\SupSet}{\raise1.75pt
     \hbox{$\,\,\scriptstyle\supset\,\,$}}
\newcommand{\SubSet}{\raise1.75pt
     \hbox{$\,\,\scriptstyle\subset\,\,$}}

\catcode`\@=11
\renewcommand{\over}{\@@over}
\renewcommand{\atop}{\@@atop}
\renewcommand{\above}{\@@above}
\renewcommand{\overwithdelims}{\@@overwithdelims}
\renewcommand{\atopwithdelims}{\@@atopwithdelims}
\renewcommand{\abovewithdelims}{\@@abovewithdelims}
\catcode`\@=13

\begin{document} 

\title{On a conjecture of Deligne}
\author{Vladimir Drinfeld}
\address{University of Chicago, Department of Mathematics, Chicago, IL 60637}
\email{drinfeld@math.uchicago.edu}
\dedicatory{Dedicated to the memory of I.M.Gelfand}

\thanks{Partially supported by NSF grant DMS-1001660}

\begin{abstract}
Let $X$ be a smooth variety over $\BF_p$. Let 
$E$ be a number field. For each nonarchimedean place
$\lambda$ of $E$ prime to $p$ consider the set of isomorphism
classes of irreducible lisse $\overline{E}_{\lambda}$-sheaves
on $X$ with determinant of finite order such that for every closed point
$x\in X$ the characteristic polynomial of the Frobenius $F_x$ has coefficents in $E$.
We prove that this set does not depend on $\lambda$.

The idea is to use a method developed by G.~Wiesend to reduce the
problem to the case where $X$ is a curve. This case was treated by
L.~Lafforgue.

\medskip

2000 Math. Subj. Classification: 14G15, 11G35.

\end{abstract}

\maketitle

\section{Introduction}   \label{s:main theorem}
 \subsection{Main theorem}   \label{ss:main theorem}
 
 \begin{theorem}  \label{t:my_result}
Let $X$ be a smooth scheme over $\BF_p$.
Let $E$ be a finite extension of $\BQ$. Let $\lambda$, $\lambda'$ be nonarchimedean places of 
$E$ prime to $p$ and $E_{\lambda}$, $E_{\lambda'}$ the corresponding completions. Let  $\E$ be a 
 lisse $\overline{E}_{\lambda'}$-sheaf on $X$ such that
 for every closed point $x\in X$ the polynomial $\det (1-F_xt, \E)$ has coefficients in $E$ and its
 roots 
 are $\lambda$-adic units.
 Then there exists a lisse $\overline{E}_{\lambda}$-sheaf on $X$ compatible with $\E$
  (i.e., having the same characteristic polynomials of the operators $F_x$ for all 
  closed points $x\in X$).
\end{theorem}
 
 According to a conjecture of Deligne (see \S\ref{ss:conj} below),
 Theorem~\ref{t:my_result} should hold for \emph{any normal scheme} of finite type over $\BF_p$.
 
 \begin{rem}   \label{r:Lafforgue}
If $\dim X=1$ then Theorem~\ref{t:my_result} is a well known corollary of the
Langlands conjecture for $GL(n)$ over functional fields proved by L.~Lafforgue \cite{La}.
More precisely, it immediately follows from \cite[Theorem VII.6]{La}.
\end{rem}

We will deduce Theorem~\ref{t:my_result} from the particular case $\dim X=1$ using
a powerful and general method developed by G.~Wiesend \cite{W1} (not long before his untimely
death).

 
  \subsection{Deligne's conjecture}  \label{ss:conj}
  Here is a  part of  \cite[Conjecture~1.2.10)]{De}.

\begin{conj}  \label{c:Deligne}
Let $X$ be a normal scheme of finite type over $\BF_p$, $\ell\ne p$ a prime, and $\E$ an irreducible lisse $\dbQbar_{\ell}$-sheaf on $X$ whose determinant has finite order. 

\begin{enumerate}
\item[(a)] There exists a subfield $E\subset\dbQbar_{\ell}$ finite over $\BQ$
such that for every closed point $x\in X$ the
polynomial $\det (1-F_xt, \E)$ has coefficients in $E$.
\item[(b)] For a possibly bigger $E$ and every nonarchimedean place $\lambda$ of
$E$ prime to $p$ there exists a lisse $E_{\lambda}$-sheaf compatible with $\E$.
\item[(c)] The roots of the polynomials $\det (1-F_xt, \E)$ (and therefore their inverses) are 
integral over $\BZ [p^{-1}]$
\end{enumerate}
\end{conj}  

In the case of curves Conjecture~\ref{c:Deligne} was completely proved by 
Lafforgue \cite[Theorem VII.6]{La}. Using a Bertini argument\footnote{The Bertini argument is clarified in \cite[\S 1.5-1.9]{De2}. A 
somewhat similar technique is explained in Appendix~\ref{s:Bertini-Poonen} below.}, he deduced from this a part of Conjecture~\ref{c:Deligne}
for $\dim X>1$. Namely, he proved statement (c) and the following part of (a):
all the coefficients of the polynomials $\det (1-F_xt, \E)$, $x\in X$, are algebraic numbers. 
The fact that the extension of $\BQ$ generated by these algebraic numbers is finite
was proved by Deligne \cite{De2, EK}. Combining (a), (c), Theorem~\ref{t:my_result}, and
the main result of \cite{Ch} one gets (b) in the case where $X$ is smooth (one needs \cite{Ch} to pass from $\overline{E}_{\lambda}$-sheaves to $E_{\lambda}$-sheaves.)

 \subsection{An open question} \label{s:open}
We will deduce Theorem~\ref{t:my_result} from the particular case $\dim X=1$ treated by 
L.~Lafforgue \cite{La} and a more technical Theorem~\ref{Mainth} (in which we consider an arbitrary regular scheme $X$ of finite type over $\BZ [\ell^{-1}]$). Following G.~Wiesend \cite{W1}, we will prove Theorem~\ref{Mainth} ``by pure thought" (see \S\ref{ss:steps} for more details).
Such a proof of Theorem~\ref{t:my_result} turns out to be possible because Wiesend's method allows to \emph{bypass} the following 
problem.

\begin{quest}   \label{q:3}
Let $X$ be an irreducible smooth variety over $\BF_q$ and let
$\E$ be an irreducible lisse $\dbQbar_{\ell}$-sheaf of rank $r$ on $X$ whose determinant has
finite order. Let $K$ be the field of rational functions on $X$ and let $\rho$ be the
$\ell$-adic representation of $\Gal (\bar K/K)$ corresponding to $\E$. Is it true that
a certain Tate twist of $\rho$ appears as a subquotient of $H^i(Y\otimes_K\bar K,\dbQbar_{\ell})$ for some algebraic variety $Y$ over $K$ and some number $i\,$?
\end{quest}

L.~Lafforgue \cite{La} gave a positive answer to Question~\ref{q:3} if $\dim X=1$. 
Without assuming that $\dim X=1$, we prove in
this article that $\rho$ is a part of a compatible system of 
representations $\rho_{\lambda}: \Gal (\bar K/K)\to GL(r,\overline{E}_{\lambda} )$, where
$E\subset \dbQbar_{\ell}$ is the number field generated by the coefficients of the polynomials 
$\det (1-F_xt, \E)$, $x\in X$, and $\lambda$ runs through the set of all
nonarchimedean places of $E$ prime to $p$. Nevertheless, {\em Question~\ref{q:3} remains open\,} if $\dim X>1$ because it is not clear how to formulate a motivic analog of the construction from
\S\ref{ss:constructing}.

\subsection{Application: the Grothendieck group of weakly motivic $\BQbar_{\ell}$-sheaves}
Theorem~\ref{t:my_result} allows to associate to each scheme $X$ of finite type over $\BF_p$
a certain group $K_{\mot}(X,\BQbar )$ (where ``mot" stands for ``motivic") and according to
a theorem of Gabber~\cite{Fuj}, the ``six operations" are well defined on $K_{\mot}$.
Details are explained in \S\ref{sss:mot1}-\ref{sss:mot3} below.

Morally, $K_{\mot}(X,\BQbar )$ should be the Grothendieck group of the ``category of motivic
$\BQbar$-sheaves" on $X$. (Here the words in quotation marks 
do not refer to any precise notion of motivic sheaf.)

\subsubsection{A corollary of Theorem~\ref{t:my_result}}  \label{sss:mot1}
 Let $X$ be a scheme of finite type over $\BF_p$. The set of its closed points will be denoted by $|X|$.
 Let $\ell$ be a prime different from
 $p$ and let  $\BQbar_{\ell}$ be an algebraic closure of $\BQ_{\ell}$.
Let $\Sh (X,\BQbar_{\ell})$ be the abelian category of $\BQbar_{\ell}$-sheaves on $X$ and
$\sD(X,\BQbar_{\ell})=D^b_c (X,\BQbar_{\ell})$ the bounded $\ell$-adic
derived category \cite[\S1.2-1.3]{De}. 
 
 Now let $\BQbar$ be an algebraic closure of $\BQ$. Suppose that we are given a map
 \begin{equation}   \label{e:Gamma}
 \Gamma: |X|\to\{\mbox{subsets of }\BQbar^{\times}\}, \quad \quad x\mapsto\Gamma_x \,.
 \end{equation}

 Once we choose a prime $\ell\ne p$, an algebraic closure $\BQbar_{\ell}\supset \BQ_{\ell}$,
 and an embedding $i:\BQbar\mono\BQbar_{\ell}$ we can consider the following
 full subcategory $\Sh_{\Gamma} (X,\BQbar_{\ell},i)\subset\Sh (X,\BQbar_{\ell})$:
 a $\BQbar_{\ell}$-sheaf $\F$ is in $\Sh_{\Gamma} (X,\BQbar_{\ell})$ if for every closed point $x\in X$ all eigenvalues of the geometric Frobenius $F_x:\F_x\to\F_x$ are in $i(\Gamma_x )$.
Let $\sD_{\Gamma}(X,\BQbar_{\ell},i)\subset \sD_{\Gamma}(X,\BQbar_{\ell})$ be the full subcategory
of complexes whose cohomology sheaves are in $\Sh_{\Gamma} (X,\BQbar_{\ell},i)$.
Let $K_{\Gamma}(X,\BQbar_{\ell},i)$ denote 
the Grothendieck group of $\Sh_{\Gamma} (X,\BQbar_{\ell},i)$, which is the same as 
the Grothendieck group of $\sD_{\Gamma} (X,\BQbar_{\ell},i)$.

For any field $E$ set
\[
A(E):=\{f\in E(t)^{\times}\;\bigl\lvert \;  f(0)=1 \}.  
\]
A sheaf $\F\in\Sh_{\Gamma} (X,\BQbar_{\ell},i)$ defines a map 
\[
f_{\F}:|X|\to A(i(\BQbar))=A(\BQbar), \quad \quad x\mapsto \det (1-F_xt,\F).
\]
For any subsheaf $\F'\subset\F$ we have  $f_{\F}=f_{\F'}f_{\F/\F'}$, so
we get a homomorphism $K_{\Gamma}(X,\BQbar_{\ell},i)\to A(\BQbar )^{|X|}$, where 
$A(\BQbar )^{|X|}$ is the group of all maps $|X|\to A(\BQbar )$. 

\begin{lem}  \label{l:stand2} 
This map is injective.
\end{lem}

\begin{proof}
We have to show that if 
$\F_1,\F_2\in\Sh_{\Gamma} (X,\BQbar_{\ell},i)$ have equal images in $A(\BQbar )^{|X|}$ then they have equal classes in $K_{\Gamma}(X,\BQbar_{\ell},i)$. Stratifying $X$, one reduces this to the
case where $\F_1,\F_2$ are lisse and $X$ is normal. Then use the \v {C}ebotarev density theorem.
\end{proof}

Lemma~\ref{l:stand2} allows to consider $K_{\Gamma}(X,\BQbar_{\ell},i)$ as a subgroup of 
$A(\BQbar )^{|X|}$. The next statement immediately follows from Theorem~\ref{t:my_result} and
Conjecture~\ref{c:Deligne}(a) proved by Deligne~\cite{De2}.

\begin{cor}   \label{c:K_0}
Suppose that for each $x\in |X|$ all elements of $\Gamma_x$ are units outside of $p$.
Then the subgroup $K_{\Gamma}(X,\BQbar_{\ell},i)\subset A(\BQbar )^{|X|}$ does not depend
on the choice of $\ell$, $\BQbar_{\ell}$, and $i:\BQbar\mono\BQbar_{\ell}$. \hfill\qedsymbol
\end{cor}


In the situation of Corollary~\ref{c:K_0} we will write simply $K_{\Gamma}(X,\BQbar )$
instead of $K_{\Gamma}(X,\BQbar_{\ell},i)$.

\subsubsection{Weakly motivic $\BQbar_{\ell}$-sheaves and their Grothendieck group}
\label{sss:mot2}

Let us consider two particular choices of the map \eqref{e:Gamma}.

\begin{defin}   \label{d:mixmot}
For $x\in |X|$ let $\Gamma_x^{\mix}\subset\BQbar^{\times}$ be the set of
 numbers $\alpha\in\BQbar^{\times}$ with the following property: there exists 
 $n\in\BZ$ such that all complex absolute values of $\alpha$ equal $q_x^{n/2}$, where
 $q_x$ is the order of the residue field of $x$. Let $\Gamma_x^{\mot}$ be the set of those numbers
 from $\Gamma_x^{\mix}$ that are units outside of $p$.
\end{defin}

In other words, $\Gamma_x^{\mot}$ is the \emph{group of  Weil numbers} with respect to $q_x$.

\medskip

Since $\Gamma_x^{\mix}$ is stable under $\Gal (\BQbar/\BQ )$ the categories
$\Sh_{\Gamma^{\mot}} (X,\BQbar_{\ell},i)$ and $\sD_{\Gamma^{\mot}} (X,\BQbar_{\ell},i)$ do not depend on the choice of $i:\BQbar\mono\BQbar_{\ell}$. We denote them by 
$\Sh_{\mot} (X,\BQbar_{\ell})$ and $\sD_{\mot} (X,\BQbar_{\ell})$. We also have similar categories  $\Sh_{\mix} (X,\BQbar_{\ell})$ and $\sD_{\mix} (X,\BQbar_{\ell})$. 

\begin{defin}  \label{d:motivic}
Objects of  $\Sh_{\mot} (X,\BQbar_{\ell})$ (resp. $\sD_{\mot} (X,\BQbar_{\ell}))$
are called \emph{weakly motivic $\BQbar_{\ell}$-sheaves} (resp. \emph{weakly motivic $\BQbar_{\ell}$-complexes).}
\end{defin}

\begin{rems}   \label{r:mixed}
\begin{enumerate}
\item[(i)] A result of L.~Lafforgue \cite[Corollary VII.8]{La} implies that $\Sh_{\mix} (X,\BQbar_{\ell})$
is equal to the category of {\it mixed $\BQbar_{\ell}$-sheaves\,} introduced by 
Deligne \cite[Definition 1.2.2]{De}. Therefore $\sD_{\mix}(X,\BQbar_{\ell})$ is equal to the category
of {\it mixed $\BQbar_{\ell}$-complexes\,} from \cite[\S6.2.2]{De}.

\item[(ii)] According to \cite{De} and \cite{BBD}, the category $\sD_{\mix}(X,\BQbar_{\ell})$ is
stable under all ``natural" functors (e.g., under Grothendieck's ``six operations" ). The same is true for $\sD_{\mot}(X,\BQbar_{\ell})$, see Appendix~\ref{s:AppendixB}.

\item[(iii)] Any indecomposable object of $\sD (X,\BQbar_{\ell})$ is a tensor product of
an object of $\sD_{\mot}(X,\BQbar_{\ell})$ and an invertible $\BQbar_{\ell}$-sheaf on
$\Spec\BF_p$ (see Theorem~\ref{t:bigoplus} from Appendix~\ref{s:AppendixB}).
\end{enumerate}
\end{rems}

Corollary~\ref{c:K_0} is applicable to $\Gamma_x^{\mot}$ (but not to $\Gamma_x^{\mix}$).
So we have a well defined group 
\[
K_{\mot}(X,\BQbar ):=K_{\Gamma^{\mot}}(X,\BQbar )\, .
\]

\begin{rem}   \label{r:CM}
Let $C$ denote the union of all CM-subfields of $\BQbar$.
For any prime power $q$, the subfield of $\BQbar$ generated by all Weil numbers with respect to $q$ equals $C$ (see Theorem~\ref{t:CM} from  Appendix~\ref{s:CM}). So for any $X\ne\emptyset$, the kernel of the action of $\Gal (\BQbar/\BQ )$ on $K_{\mot}(X,\BQbar )$ equals 
$\Gal (\BQbar/C)$.
\end{rem}

\subsubsection{Functoriality of $K_{\mot}(X,\BQbar )$} \label{sss:mot3}
By Remark~\ref{r:mixed}(ii), the ``six operations" preserve the class of
weakly motivic $\BQbar_{\ell}$-complexes. So it is clear that once you fix a prime $\ell\ne p$, 
an algebraic closure $\BQbar_{\ell}\supset\BQ_{\ell}$, 
and an embedding $i:\BQbar\mono\BQbar_{\ell}$, you get an action of the 
``six operations" on $K_{\mot}$.
O.~Gabber proved~\cite[Theorem 2]{Fuj} that in fact, \emph{the action of the
``six operations" on $K_{\mot}$ does not depend on the
choice of $\ell$, $\BQbar_{\ell}$, and $i$.} 
By virtue of Theorem~\ref{t:my_result} and
Conjecture~\ref{c:Deligne}(a) proved by Deligne, another result of
Gabber ~\cite[Theorem 3]{Fuj} can be reformulated as follows: \emph{the basis of 
$K_{\mot} (X,\BQbar )$ formed by the classes of irreducible perverse sheaves is independent of
$\ell$, $\BQbar_{\ell}$, and $i$.}

\subsection{Structure of the article}    \label{ss:structure}
In \S\ref{ss:LS}-\ref{ss:main technical} we formulate Theorem~\ref{Mainth}, which is the main technical 
result of this article. It gives a criterion for the existence of a lisse $E_{\lambda}$-sheaf on a
regular scheme $X$ of finite type over $\BZ [\ell ^{-1}]$ with prescribed polynomials
$\det (1-F_xt, \E)$, $x\in X$; the criterion is formulated in terms of 1-dimensional subschemes of
$X$. In \S\ref{ss:implies} we deduce Theorem~\ref{t:my_result} from Theorem~\ref{Mainth}. 

In \S\ref{ss:steps} we formulate three propositions which imply Theorem~\ref{Mainth}.
They are proved in \S\ref{s:proof1}-\ref{s:proofc}. \

Following Wiesend \cite{W1}, we use the Hilbert irreducibility theorem as the main technical
tool. A variant of this theorem convenient for our purposes is formulated in \S\ref{ss:Hilbert} 
(see Theorem~ \ref{t:Hilbert}) and proved in Appendix~\ref{s:Hilbert}.
In the case of schemes over $\BF_p$ one can use Bertini theorems instead of Hilbert irreducibility (under a tameness assumption, this is explained in \S\ref{ss:char_p} and Appendix~\ref{s:Bertini-Poonen}).

In \S\ref{s:counterex} we give counterexamples showing that in
Theorems~\ref{Mainth}  and \ref{t:Hilbert} the regularity assumption cannot be replaced by normality.

In Appendix~\ref{s:AppendixB} we show that the category of weakly motivic $\BQbar_{\ell}$-sheaves,
$\sD_{\mot}(X,\BQbar_{\ell})$, defined in \S\ref{sss:mot2} is stable under all ``natural" functors.

In Appendix~\ref{s:Bertini-Poonen} we discuss the Bertini theorem and Poonen's 
``Bertini theorem over finite fields".

In Appendix~ \ref{s:CM} we justify Remark~\ref{r:CM} by proving that the union of all CM fields is generated by Weil numbers.

\bigskip
I thank P.~ Deligne for sending me his letter (March 5, 2007) with a proof of 
Conjecture~\ref{c:Deligne}(a). I also thank A.~Beilinson,  B.~Conrad, H.~Esnault, O.~Gabber,  D.~Kazhdan, M.~Kerz, and M.~Kisin for useful discussions, advice, and remarks. In particular, Beilinson communicated to me the counterexample from \S\ref{ss:Sasha} and Kerz communicated to me a simple proof of Proposition~\ref{p:b}.

\section{Formulation of the main technical theorem}   \label{s:formulation}

Fix a prime $\ell $ and a finite extension $E_{\lambda}\supset\dbQ_{\ell}$.
Let $O\subset E_{\lambda}$ denote its ring of integers.

\subsection{The sets $\LS_r (X)$ and $\widetilde{\LS}_r(X)$}   \label{ss:LS}
\subsubsection{The set\, $\LS_r (X)$}   \label{sss:LS}
Let $X$ be a scheme  of finite type over $\BZ [\ell^{-1}]$. 
Say that  lisse $E_{\lambda}$-sheaves on $X$ are \emph{equivalent} if they have isomorphic semisimplifications. Let $\LS_r (X)=\LS_r^{E_{\lambda}} (X)$ be the set of equivalence classes of
lisse $E_{\lambda}$-sheaves on $X$ of rank $r$. Clearly $\LS_r (X)$ is a contravariant functor
in $X$.

\begin{example}    \label{ex:point}
Suppose that $X$ has a single point $x$. If $\E$ is an $E_{\lambda}$-sheaf on $X$ of rank $r$ then
$\det (1-F_xt,\E)$ is a polynomial in $t$ of the form
\begin{equation} \label{Polr}
1+c_1t+\ldots +c_rt^r, \quad  c_i\in O, \; c_r\in O^{\times}.
\end{equation}
Let $P_r(O)$ be the set of all polynomials of the from \eqref{Polr}. Thus we get a bijection
$\LS_r (X)\iso P_r(O)$.
\end{example}

\subsubsection{The sets $|X|$ and $||X||$} Let $X$ be a scheme  of finite type over $\BZ$.
We write $|X|$ for the set of closed points of $X$.
Let $||X||$ denote the set of isomorphism classes of pairs consisting of a finite field $\BF$ and a morphism $\alpha :\Spec\BF\to X$. Associating to such $\alpha$ the point $\alpha (\Spec\BF)\in |X|$
one gets a canonical map  $||X||\to |X|$. The multiplicative monoid of positive integers, 
$\BN$, acts on $||X||$ (namely, $n\in\BN$ acts by replacing a finite field $\BF$ with its extension of degree $n$). Both $||X||$ and $|X|$ depend functorially on $X$. The map $||X||\to |X|$ and the action of
$\BN$ on $||X||$ are functorial in $X$.

\begin{rem}   \label{r:nonfunctorial}
One 
has a canonical embedding $|X|\mono ||X||$ (given $x\in |X|$ take $\BF$ to be the residue field of $x$ and take $\alpha :\Spec\BF\to X$ to be the canonical embedding). Combining the 
embedding $|X|\mono ||X||$ with the action of $\BN$ on $||X||$ we get a bijection 
$\BN\times |X|\iso ||X||$. Note that the embedding $|X|\mono ||X||$ and the bijection 
$\BN\times |X|\iso ||X||$ are {\em not functorial\,} in $X$.
\end{rem}

\subsubsection{The set\, $\widetilde{\LS}_r(X)$}   \label{sss:tildeLS}
The set $P_r (O)$ from Example~\ref{ex:point} can be thought of
as the set of $O$-points of a scheme $P_r$ (which is isomorphic to
$(\BG_a)^{r-1}\times\BG_m$). The morphism  $(\BG_m )^r\to P_r$ that takes 
$(\beta_1 ,\ldots,\beta_r )$ to the polynomial $\prod\limits_{i=1}^r (1-\beta_i t)$ induces an isomorphism $(\BG_m )^r/S_r\iso P_r$, where $S_r$ is the symmetric group.
The ring homomorphism $\BZ\to\End ((\BG_m )^r)$ defines an
action of the multiplicative monoid $\BN$ on $(\BG_m )^r$ and therefore an action of
$\BN$ on $P_r$. Now let $X$ be a scheme  of finite type over $\BZ$.

\begin{defin}    \label{d:tildeLS}
$\widetilde{\LS}_r (X)$ (or $\widetilde{\LS}_r^{E_{\lambda}} (X)$)
 is the set of $\BN$-equivaraint maps $||X||\to P_r (O)$.
\end{defin}

\begin{rem}    \label{r:alt_tildeLS}
By Remark~\ref{r:nonfunctorial}, restricting an  $\BN$-equivaraint map $||X||\to P_r (O)$ to
the subset $|X|\subset ||X||$ one gets a bijection
\begin{equation}    \label{e:nonfunctorial}
\widetilde{\LS}_r (X)\iso \{ \mbox{Maps from } |X| \mbox { to } P_r (O) \} .
\end{equation}
It is not functorial in $X$ if the structure of functor on the r.h.s. of
\eqref{e:nonfunctorial} is introduced naively. Nevertheless, we will use \eqref{e:nonfunctorial} 
in order to write elements of $\widetilde{\LS}_r (X)$ as maps $|X|\to  P_r (O)$. The value of
$f\in\widetilde{\LS}_r (X)$ at $x\in |X|$ will be denoted by $f(x)$ or $f_x$ (the latter allows to write
the corresponding polynomial \eqref{Polr} as $f_x(t)\,$).
\end{rem}

\subsubsection{The map\, $\LS_r(X)\to\widetilde{\LS}_r(X)$}   \label{sss:tildeLStoLS}
Let $X$ be a scheme  of finite type over $\BZ [\ell^{-1}]$.
A lisse $E_{\lambda}$-sheaf $\E$ on $X$ defines an $\BN$-equivaraint map $f_{\E}: ||X||\to P_r (O)$: 
namely, if $\BF$ is a finite field equipped with a morphism $\alpha :\Spec\BF\to X$ then
 $f_{\E} (\BF ,\alpha )\in P_r (O)$ is the characteristic polynomial of the geometric Frobenius
 with respect to $\BF$ acting on the stalk of $\alpha^*\E$.
Thus we get a map\, $\LS_r(X)\to\widetilde{\LS}_r(X)$ functorial in $X$.


If $X_{\red}$ is normal this map is injective by the 
\v {C}ebotarev density theorem (see \cite[Theorem 7]{S}). 
In this case we do not distinguish an element of $\LS_r(X)$ from its image in 
$\widetilde{\LS}_r(X)$ and consider $\LS_r(X)$ as a subset of $\widetilde{\LS}_r(X)$.
%
%



\subsection{Formulation of the theorem}  \label{ss:main technical} 

If $X$ is a separated curve over a field\footnote{According to \cite{W2,KS2}, there is a good notion of tameness in a much more general situation.} there is a well known notion of tame etale covering and therefore a notion of tame lisse $E_{\lambda}$-sheaf (``tame" means ``tamely ramified at infinity"). If two
lisse $E_{\lambda}$-sheaves are equivalent (i.e., have isomorphic  semisimplifications) and one of them is tame then so is the other.
Let  $\LS_r^{\tame} (X)\subset\LS_r (X)$ be the subset of equivalence classes 
of tame lisse $E_{\lambda}$-sheaves. 

%

By an {\em arithmetic curve\,} we mean a scheme of finite type over $\BZ$ of pure dimension $1$.

\begin{theorem}   \label{Mainth}
Let $X$ be a regular scheme of finite type over $\BZ [\ell^{-1}]$. An element 
$f\in\widetilde{\LS}_r(X)$ belongs to $\LS_r(X)$ if and only if it satisfies the following conditions:

\begin{enumerate}
\item[(i)] for every regular arithmetic curve $C$ and every morphism 
$\varphi :C\to X$ one has $\varphi^*(f)\in\LS_r(C)$;
\item[(ii)]  there exists a dominant etale morphism $X'\to X$ 
such that for every smooth separated curve $C$ over a finite field and every morphism 
$C\to X'$ the image of $f$ in $\widetilde{\LS}_r (C)$ belongs to $\LS_r^{\tame} (C)$.
\end{enumerate}
\end{theorem}

\begin{rems}  \label{r:Sasha's counterex}
\begin{enumerate}
\item[(i)] In Theorem~\ref{Mainth} the regularity assumption on $X$ cannot be replaced by normality
(in \S\ref{s:counterex} we give two counterexamples, in which $X$ is a surface over a finite field).
The regularity assumption allows us to use the Zariski-Nagata purity theorem in the proof of
Corollary~\ref{c:specialization} and to apply Theorem~\ref{t:Hilbert}.

\item[(ii)]  I do not know if the regularity assumption in Theorem~\ref{t:my_result}  can be replaced by normality.

\item[(iii)] The sets $\LS_r (X)$ and $\widetilde{\LS}_r(X)$ were defined in \S\ref{ss:LS}
for a fixed finite extension $E_{\lambda}\supset\dbQ_{\ell}$, so 
$\LS_r (X)=\LS_r ^{E_{\lambda}}(X)$, $\widetilde{\LS}_r (X)=\widetilde{\LS}_r ^{E_{\lambda}}(X)$.
Replacing $E_{\lambda}$ by $\dbQbar_{\ell}$ in these definitions one gets bigger sets, denoted by $\LLS_r (X)$ and $\widetilde{\LLS}_r (X)$. {\em If $X$ is a smooth scheme over $\BF_p$ then
Theorem~\ref{Mainth} remains valid for $\LLS_r (X)$ and $\widetilde{\LLS}_r (X)$ instead of 
$\LS_r (X)$ and $\widetilde{\LS}_r(X)$.} This follows from Theorem~\ref{Mainth} as stated above combined with \cite[Remark~3.10]{De2} and Lemma~\ref{l:Brauer} below (the remark and
the lemma ensure that an element of $\widetilde{\LLS}_r (X)$ satisfying conditions (i-ii) from
Theorem~\ref{Mainth} belongs to $\widetilde{\LS}_r ^{E_{\lambda}}(X)$ for some subfield
$E_{\lambda}\subset\dbQbar_{\ell}$ finite over $\dbQ_{\ell}$).

\item[(iv)] Theorem~\ref{Mainth} and its proof remain valid for regular \emph{algebraic spaces}
of  finite type over $\BZ [\ell^{-1}]$.
\end{enumerate}
\end{rems}

\subsection{Theorem~\ref{Mainth} implies Theorem~\ref{t:my_result}}   \label{ss:implies}
\begin{lem}   \label{l:Brauer}
Let $G$ be a group. Let $\rho$ be a semisimple representation of $G$ over 
$\overline{E}_{\lambda}$ of dimension $r<\infty$ whose character is defined over $E_{\lambda}$.
Let $F\subset\overline{E}_{\lambda}$ be any extension of $E_{\lambda}$ such that 
$[F:E_{\lambda}]$ is divisible by $r$, $r-1$,\ldots, $2$. Then $\rho$ can be defined over 
$F$.
\end{lem}

\begin{proof}
We can assume that $\rho$ cannot be decomposed into a direct sum of representations of
dimension $<r$ whose characters are defined over $E_{\lambda}$. Then 
$\rho =\bigoplus\limits_{i\in I}\rho_i$, where the $
\rho_i$'s are irreducible, $\rho_i\not\simeq\rho_j$ for $i\ne j$, and the action of
$\Gal (\overline{E}_{\lambda}/E_{\lambda})$ on $I$ is transitive. Clearly $\dim \rho_i=r':=r/d$,
where $d=\Card (I)$. Let us show that for every $\Gal (\overline{E}_{\lambda}/F)$-orbit 
$\cO\subset I$, the representation $\rho_{\cO} =\bigoplus\limits_{i\in\cO}\rho_i$ can be defined over~$F$. Fix $i_0\in\cO$, then the stabilizer of $i_0$ in 
$\Gal (\overline{E}_{\lambda}/E_{\lambda})$ equals $\Gal (\overline{E}_{\lambda}/K)$ for some
extension $K\supset E_{\lambda}$ of degree $d$. The character of $\rho_{i_0}$ is defined over
$K$. The obstruction to $\rho_{i_0}$ being defined over $K$ is an element $u\in\Br (K)$ with
$r'u=0$. We claim that
\begin{equation} \label{e:u_killed}
u\in\Ker (\Br (K)\to\Br (KF)),
\end{equation}
where $KF\subset \overline{E}_{\lambda}$ is the composite field. To prove this, it suffices to check that
\begin{equation} \label{e:divisibility}
r'\, |\, [KF:K]  \; .
\end{equation}
But $r\,|\, [F:E_{\lambda}]$ (by the assumption on $F$), so $r\,|\, [KF:E_{\lambda}]$, which is 
equivalent to \eqref{e:divisibility}. By \eqref{e:u_killed}, $\rho_{i_0}$ is defined over $KF$, i.e., 
$\rho_{i_0}\simeq V\otimes_{KF}\overline{E}_{\lambda}$, where $V$ is a representation of
$G$ over $KF$. Then $V\otimes_K\overline{E}_{\lambda}=
V\otimes_{KF}(KF\otimes_K\overline{E}_{\lambda})\simeq\rho_{\cO}$, so $\rho_{\cO}$ is
defined over $K$.
\end{proof}

Now let us deduce Theorem~\ref{t:my_result} from Theorem~\ref{Mainth}.
Let $E$, $E_{\lambda}$, $E_{\lambda'}$, and $\E$ be as in Theorem~\ref{t:my_result}.
Then $\E$ defines an element $f\in\widetilde{\LS}_r^{E_{\lambda}}(X)$ and the problem is to show that $f\in\LS_r^F(X)$ for some finite extension $F\supset E_{\lambda}$. Let $F\supset E_{\lambda}$ be any extension of degree $r!$ . Let us show that $f\in\widetilde{\LS}_r^{F}(X)$ satisfies conditions 
(i)-(ii) of Theorem~\ref{Mainth}. Condition (i) follows from Theorem~\ref{t:my_result} for curves (proved by L.~Lafforgue) combined with Lemma~\ref{l:Brauer}.
 Condition (ii) holds for $f$ viewed as an element of 
$\LS_r^{E_{\lambda'}}(X)$. To conclude that it holds for $f$ viewed as an element of 
$\widetilde{\LS}_r^{F}(X)$, use the following corollary of
 \cite[Theorem 9.8]{De0}: if $C$ is a smooth curve over a finite field, $\F'$ is a tame lisse 
 $E_{\lambda'}$-sheaf on $C$, and $\F$ is a lisse $F$-sheaf on $C$ compatible with
 $\F'$ then $\F$ is also tame. Now we can apply Theorem~\ref{Mainth} and get 
 Theorem~\ref{t:my_result}.

\subsection{Steps of the proof of Theorem~\ref{Mainth}} \label{ss:steps}
We follow Wiesend's work \cite{W1} (see also the related ar\-ticle~\cite{KS1}).  

The ``only if" statement of Theorem~\ref{Mainth} is easy. 
First of all, if $\E$ is a torsion-free lisse $O$-sheaf
of rank $r$ then the corresponding $f\in\LS_r (X)$ clearly satisfies (i).
Property (ii) also holds for $f$: choose $X'$ so that $\E/\ell\E$ is trivial and use the fact
that the kernel of the homomorphism $GL(r,O)\to GL(r,O/\ell O)$ is a pro-$\ell$-group, so
it cannot contain nontrivial pro-$p$-subgroups  for $p\ne\ell$.

The ``if" statement of Theorem~\ref{Mainth} follows from Propositions~\ref{p:a}-\ref{p:c} formulated below. 

\begin{defin}  \label{d:amost-pro-ell}
A pro-finite group is said to be \emph{almost pro-$\ell$} if it has an open pro-$\ell$-subgroup.
\end{defin}

\begin{rem}  \label{r:amost-pro-ell}
In this case the open pro-$\ell$-subgroup can be chosen to be normal.
\end{rem}

\begin{defin}   \label{d:triviality_mod_I}
We say that $f\in\widetilde{\LS}_r(X)$ is {\em trivial modulo an ideal $I\subset O$\,}
if for every $x\in |X|$ the polynomial 
$f_x\in O[t]$ is congruent to $(1-t)^r$ modulo $I$. 
\end{defin}



\begin{defin}   \label{d:LS'}
If $X$ is connected then $\LS'_r(X)$ is the set of all $f\in\widetilde{\LS}_r(X)$ satisfying condition (i) from
Theorem~\ref{Mainth} and the following one: there is a closed normal subgroup $H\subset \pi_1 (X)$ such that 
\begin{enumerate}
\item[(a)] $\pi_1 (X)/H$ is almost pro-$\ell$;
\item[(b)]  for every nonzero ideal $I\subset O$ there is an open subgroup $V\subset\pi_1 (X)$ containing $H$
such that the pullback of $f$ to $X_V$ is trivial modulo~$I$ (here $X_V$ is the connected covering of $X$
corresponding to $V$).
\end{enumerate}
If $X$ is disconnected then $\LS'_r(X)$ is the set of all $f\in\widetilde{\LS}_r(X)$ whose restriction to each
connected component $X_{\alpha}\subset X$ belongs to $\LS'_r(X_{\alpha})$.
\end{defin}

The image of $\LS_r(X)$ in $\widetilde{\LS}_r(X)$ is clearly contained in $\LS'_r(X)$.

\begin{prop} \label{p:a}
Let $X$ be a scheme of finite type over $\BZ [\ell^{-1}]$. Suppose that 
$f\in\widetilde{\LS}_r(X)$ satisfies conditions (i)-(ii) from Theorem~\ref{Mainth}.
Then there is a dense open $U\subset X$ such that  the image of $f$ in $\widetilde{\LS}_r(U)$
belongs to $\LS'_r(U)$.
\end{prop}

Proposition~\ref{p:a} will be proved in \S\ref{s:proof1} using only 
standard facts about  fundamental groups. The next statement is the key step of the proof
of Theorem~\ref{Mainth}.

\begin{prop} \label{p:b}
Let $X$ be a regular scheme of finite type over $\BZ [\ell^{-1}]$.
Then $\LS'_r(X)=\LS_r(X)$.  
\end{prop}

\begin{prop} \label{p:c}
Let $X$ be a regular scheme of finite type over $\BZ [\ell^{-1}]$. Suppose that $f\in\widetilde{\LS}_r(X)$ satisfies condition (i) of Theorem~\ref{Mainth}. If there exists a dense open $U\subset X$ such that $f|_U\in\LS_r(U)$ then $f\in\LS_r(X)$.
\end{prop}

Propositions~\ref{p:b} and ~\ref{p:c} will be proved in \S\ref{s:proofb} and 
\S\ref{s:proofc} using the Hilbert irreducibility theorem. 

\subsection{Hilbert irreducibility}  \label{ss:Hilbert}
We prefer the following formulation of Hilbert irreducibility, which is very close to
\cite[Lemma 20]{W1} or \cite[Proposition 1.5]{KS1}.

\begin{theorem}  \label{t:Hilbert}
Let $X$ be an irreducible regular scheme of finite type over $\BZ$, $U\subset X$ a non-empty
open subset, $H\subset\pi_1 (U)$ an open normal subgroup, and $S\subset X$ a finite reduced
subscheme.
\begin{enumerate}
\item[(i)] There exists an irreducible regular arithmetic curve $C$ with a morphism 
$\varphi :C\to X$ and a section $\sigma :S\to C$ such that $\varphi (C)\cap U\ne\emptyset$
and the homomorphism $\pi_1(\varphi^{-1}(U))\to\pi_1 (U)/H$ is surjective.
\item[(ii)]  Suppose that for each $s\in S$ we are given a 1-dimensional subspace 
$l_s\subset T_sX$, where $T_sX=(\fm_s/\fm_s^2)^*$ is the tangent space. Then one can find
$C$, $\varphi$, $\sigma$ as above so that for each $s\in S$ one has\,
$\im (T_{\sigma (s)}C\to T_sX)=l_s$.
\end{enumerate}
\end{theorem}
In Appendix~\ref{s:Hilbert} we deduce Theorem~\ref{t:Hilbert} from a conventional formulation of
Hilbert irreducibility.

\begin{rems}  \label{r:regularity_2ess}
\begin{enumerate}
\item[(i)] In Theorem~\ref{t:Hilbert} the regularity assumption 
cannot be replaced by normality, see Remark~\ref{r:regularity_ess} at the end of \S\ref{s:counterex}.
\item[(ii)] Theorem~\ref{t:Hilbert} and its proof remain valid for regular \emph{algebraic spaces}
of  finite type over $\BZ$.
\end{enumerate}
\end{rems}

\begin{prop}    \label{p:Hilbert-pro-l}
Let $\ell$ be a prime such that $X\otimes\BZ [\ell^{-1}]\ne\emptyset$.
Then Theorem~\ref{t:Hilbert} remains valid for every closed normal subgroup
$H\subset\pi_1(U)$ such that $\pi_1(U)/H$ is almost pro-$\ell$.
\end{prop}

\begin{proof}
Set $G:=\pi_1(U)/H$. It suffices to find an open normal subgroup 
$V\subset G$ with the following property: every closed subgroup $K\subset G$ such that
the map $K\to G/V$ is surjective equals $G$ (then one can apply Theorem~\ref{t:Hilbert} to
$V$ instead of $H$).

Let $V'\subset G$ be an open normal pro-$\ell$-subgroup and let
$V\subset V'$ be the normal subgroup such that $V'/V=H_1 (V',\BZ/\ell\BZ)$.
Then $V$ has the required properties. Let us check that $V$ is open.
Let $\tilde U \to U$ be the covering corresponding to $V'$. Set 
$\tilde U[\ell^{-1}]:=\tilde U\otimes\BZ [\ell^{-1}]$. Then 
\[
H^1(V',\BZ/\ell\BZ)\subset H^1 (\tilde U, \BZ/\ell\BZ) \subset H^1 (\tilde U[\ell^{-1}], \BZ/\ell\BZ).
\]
The group $H^1(\tilde U[\ell^{-1}] ,\BZ/\ell\BZ)$ is finite (see \cite[ch.II, Proposition~7.1]{Milne}). 
Therefore  $H^1(V',\BZ/\ell\BZ)$ is finite.
Thus the group $V'/V=H_1(V',\BZ/\ell\BZ)$ is finite and therefore $V$ is open.
\end{proof}

\subsection{The characteristic $p$ case}
 \label{ss:char_p}
 
\begin{rems}   \label{r:Bertini} 
(i) If $X$ is an irreducible smooth projective scheme over $\BF_p$ and $U=X$ then
Theorem~\ref{t:Hilbert} remains valid for $H=\{ 1 \}$. This follows from
B.~Poonen's ``Bertini theorem over finite fields" \cite[Theorems 1.1-1.2]{Po} combined with the usual Bertini theorem  \cite[Ch.III, Corollary 7.9]{Ha}.

(ii) More generally, suppose that $X$ is a quasiprojective irreducible smooth  scheme over 
$\BF_p$ and $U\subset X$ any non-empty open subset. Represent $X$ as an
open subvariety of an irreducible normal projective variety $\bar X$ over $\BF_p$. Let
$D_1,\ldots ,D_n$ be the irreducible components of $\bar X\setminus U$ of dimension $\dim X-1$.
Etale coverings of $U$ tamely ramified at the generic points of $D_1,\ldots ,D_n$ are classified
by $\pi_1(U)/H$, where $H$ is a certain normal subgroup of $\pi_1(U)$. Then 
Theorem~\ref{t:Hilbert} remains valid for this~$H$. This follows from 
Proposition~\ref{p:justifying_Bertini-remark} (see Appendix~\ref{s:Bertini-Poonen}).
\end{rems}

\begin{rem}   \label{r:modifications} 
If $X$ is a manifold over $\BF_p$ one can slightly modify the proof of Theorem~\ref{Mainth}.
First, by Remarks~\ref{r:Bertini}, it suffices to prove a weaker version of Proposition~\ref{p:a},
with $\LS'_r(U)$ being replaced by the set of all $f\in\widetilde{\LS}_r(U)$ satisfying the following condition: for every nonzero ideal $I\subset O$
there exists a surjective finite etale morphism $\pi: U'\to U$ such that $\pi^*f$
is trivial modulo~$I$. Another possible modification 
is indicated in Remark~\ref{r:geom_semisimplicity} below.  
\end{rem}

\section{Proof of Proposition~\ref{p:a} (after G.~Wiesend)} \label{s:proof1}
Proposition~\ref{p:a} clearly follows from the next lemma, in which condition (ii) 
is stronger than condition (ii) from Theorem~\ref{Mainth}. The lemma and its proof only slightly differs from 
\cite[Proposition 17]{W1} or \cite[Proposition 3.6]{KS1}.

\begin{lem}   \label{lem1}
Let $X$ be a scheme of finite type over $\BZ [\ell^{-1}]$ and $f\in\widetilde{\LS}_r(X)$.
Assume that 
\begin{enumerate}
\item[(i)] for every regular arithmetic curve $C$ and every morphism 
$\varphi :C\to X$ one has $\varphi^*(f)\in\LS_r(C)$;
\item[(ii)]  if $C$ is a separated curve over a finite field then $\varphi^*(f)\in\LS_r^{\tame}(C)$.
\end{enumerate}
Then there is a dense open $U\subset X$ such that  the image of $f$ in $\widetilde{\LS}_r(U)$
belongs to $\LS'_r(U)$.
\end{lem}

\begin{proof}
We can assume that $X$ is reduced, irreducible, and normal. If $\dim X\le 1$ the statement is obvious.
Now assume that $\dim X>1$ and the lemma holds for all schemes whose dimension is less than that of  $X$. Replacing $X$ by an open subset we can assume that there is a smooth morphism from $X$ to some scheme $S$ such that the geometric fibers of the morphism are non-empty connected curves.
After shrinking $S$ we can assume that one of the following holds:

\begin{enumerate}
\item[(i)]  the morphism $X\to S$ has a factorization
\[
X=\bar X\setminus D\subset\bar X {\buildrel{\pi}\over{\longrightarrow}} S,
\]
where $\pi$ is a smooth projective morphism, $X\setminus D$ is fiberwise dense in $X$,
and $D$ is finite and etale over~$S$;

\item[(ii)]  $S$ is a scheme over $\BF_p$ and for some $n\in\BN$ the morphism $X^{(p^n)}\to S$
has a factorization as in (i)
(here $X^{(p^n)}$ is obtained from $X$ by base change with respect to $\Fr^n:S\to S$).
\end{enumerate}

If we are in situation (ii) then it suffices to prove the statement for $X^{(p^n)}$ instead of $X$.
So we can assume that we are in situation (i). 
(Another way to conclude that is suffices to consider situation (i) is to use M.~Artin's theorem
on the existence of ``elementary fibrations" \cite[expos\'e XI, Propostion 3.3]{SGA4}.)

We can also assume that the morphism $X\to S$ has a section $\sigma :S\to X$ and that
$\sigma^* (f)\in \LS'_r(S)$ (otherwise replace $S$ by $S'$ and
$X$ by $X\times_SS'$, where $S'$ is an appropriate scheme etale over $S$). 

Let $\fm\subset O$ be the maximal ideal. For every $n\in\BN$ set $G_n=GL(r,O/\fm^n)$ and
consider the functor that associates to an $S$-scheme $S'$ the set of isomorphism classes of tame\footnote{Here ``tame" means ``tamely ramified along $D$ relatively to $S\,$".} $G_n$-torsors on $X\times_SS'$ trivialized over $S' \buildrel{\sigma'}\over{\mono}X\times_SS'$. It follows from \cite[expos\'e~XIII, Corollaries 2.8-2.9]{SGA1} that 
 \begin{enumerate}
\item[(i)] this functor is representable by a scheme $T_n$ etale and of finite type over $S$,
\item[(ii)]  the morphism $T_{n+1}\to T_n$ is finite for each $n\ge 1$. 
\end{enumerate}
So after shrinking $S$ we can assume that the morphism $T_n\to S$ is finite for each $n$.
We will prove that in this situation $f\in\LS'_r(X)$.

By Definition~\ref{d:LS'}, to show that $f\in\LS'_r(X)$ we have to construct surjective finite etale morphisms $X_n\to X$, $n\in\BN$, so that
 \begin{enumerate}
\item[(a)] the image of $f$ in $\widetilde{\LS}_r(X_n)$ is trivial modulo~$\fm^n$,
\item[(b)] for some (or any) geometric point $\bar x$ of $X$, the quotient of the group 
$\pi_1 (X,\bar x)$ by the intersection of the kernels of its actions on the fibers $(X_n)_{\bar x}$, $n\in\BN$, is almost pro-$\ell$. 
\end{enumerate}

Since $\sigma^* (f)\in \LS'_r(S)$ we have a sequence of surjective finite etale morphisms 
$S_n\to S$, $n\in\BN$, that satisfies the analogs of (a) and (b) for $S$.
On the other hand, we will construct surjective finite etale morphisms $Y_n\to X$, $n\in\BN$, with the following properties: 
 \begin{enumerate}
\item[$(*)$]
for every geometric point $\bar s\to S$,
every tame locally constant sheaf of $r$-dimensional free $(O/\fm^n)$-modules on the fiber 
$X_{\bar s}$ has constant pullback to $(Y_n)_{\bar s}$;
\item[$(**)$] 
for some (or any) geometric point $\bar x$ of $X$, the quotient of the group $\pi_1 (X,\bar x)$ by the intersection of the kernels of its actions on the fibers $(Y_n)_{\bar x}$, $n\in\BN$, is almost pro-$\ell$. 
\end{enumerate}
Then we can take $X_n:=S_n\times_SY_n$. 

To construct $Y_n$, consider the universal $G_n$-torsor $\fT_n\to T_n\times_SX$.
Now set $Y_n:=\Res (\fT_n)$, where 
\[
\Res : \{\mbox{schemes over }T_n\times_SX\}\to \{\mbox{schemes over }X\}
\]
is the Weil restriction functor. In other words, $Y_n$ is the scheme over $X$ such that for any 
$X$-scheme $X'$ one has 
\[
\Mor_X(X',Y_n):=\Mor_{T_n\times_SX}(T_n\times_SX', \fT_n).
\]
The fiber of $Y_n$ over any geometric point $\bar s\to S$ equals the fiber product 
of the $G_n$-torsors over $X_{\bar s}$ corresponding to all points of $(\fT_n)_{\bar s}$, so
$Y_n$ has property $(*)$. We will show that property $(**)$ also holds.

Let $\eta\in S$ be the generic point and $\bar\eta\to\eta$ a geometric point. 
In property $(**)$ we take $\bar x$ to be the composition
$\bar\eta\to\eta\mono S\buildrel{\sigma}\over{\longrightarrow}  X$.
Let $\pi_1^t (X_{\eta},\bar x)$ denote the tame fundamental group (``tame" means
``tamely ramified along $D_{\eta}$"). Let $H_n\subset \pi_1^t (X_{\eta},\bar x)$
be the kernel of the action of $\pi_1^t (X_{\eta},\bar x)$ on $(X_n)_{\bar x}$.
Since $X$ was assumed normal the map
$\pi_1 (X_{\eta},\bar x)\to\pi_1 (X,\bar x)$ is surjective, so to prove
property $(**)$ it suffices to check that the quotient 
$\pi_1^t (X_{\eta},\bar x)/\bigcap\limits_n H_n$ is almost pro-$\ell$.

We have an exact sequence
\[
0\to K\to\pi_1^t (X_{\eta},\bar x)\to\Gamma\to 0, \quad 
K:=\pi_1^t (X_{\bar\eta},\bar x), \quad \Gamma:=\pi_1 (\eta ,\bar\eta ) 
\]
and a splitting $\sigma_*:\Gamma\to\pi_1^t (X_{\eta},\bar x)$. Thus 
$\pi_1^t (X_{\eta},\bar x)$ identifies with the semidirect product
$\Gamma\ltimes K$. The subgroup $H_n\subset \pi_1^t (X_{\eta},\bar x)$ identifies
with $\Gamma_n\ltimes K_n$, where $K_n$ is the intersection of the kernels of all homomorphisms $K\to GL(r, O/\fm^n )$ and $\Gamma_n$ is the kernel of the action of $\Gamma$
on $K/K_n$.

The group $K$ is topologically finitely generated, so the group $K':=K/\bigcap\limits_n K_n$ is almost pro-$\ell$. Thus it remains to show that the quotient $\Gamma/\bigcap\limits_n \Gamma_n$
is almost pro-$\ell$. Since $\bigcap\limits_n \Gamma_n\supset\Ker (\Gamma\to\Aut K')$ it
suffices to show that the group $\Aut K'$ equipped with the compact-open topology is
an almost pro-$\ell$-group. This follows from the fact that $K'$ is topologically finitely generated
and almost pro-$\ell$. Indeed, if $V\subset K'$ is the maximal open normal pro-$\ell$-subgroup
then the automorphisms of $K'$ that act as identity on the finite groups
$K'/V$ and $H_1 (V,\BZ/\ell\BZ)$ form an open pro-$\ell$-subgroup of $\Aut K'$.
\end{proof}

\begin{rem}
Suppose that $X$ is irreducible. Lemma~\ref{lem1} says that for each $f\in\widetilde{\LS}_r(X)$
satisfying certain conditions there exists a non-empty open $U\subset X$ and a normal subgroup
$H\subset\pi_1 (X)$ such that $f|_U$ saitisfies the condition from
Definition~\ref{d:LS'}. The proof of Lemma~\ref{lem1} shows that one can choose $U$ and $H$
to be independent of $f$. This is not surprising: Theorem~\ref{Mainth} will show that one can
take $U$ to be the set of regular points of $X_{\red}$ and $H$ to be the intersection of the kernels of all homomorphisms $\pi_1(X)\to GL(r,O)$ satisfying a tameness condition.
\end{rem}

\section{Proof of Proposition~\ref{p:b} (after Moritz Kerz)}    \label{s:proofb}
We can assume that $X$ is irreducible.
The problem is to show that if $f\in\LS'_r(X)$ then $f\in\LS_r(X)$, i.e.,
$f$ comes from a representation $\rho :\pi_1 (X)\to GL(r,E_{\lambda})$. 
The original proof of Proposition~\ref{p:b} can be found in version 5 of the 
e-print \cite{Dr} (the idea was to first construct the character of $\rho$ using elementary representation theory and a compactness argument). The simpler proof given in 
\S\ref{ss:constructing}-\ref{ss:Faltings} is due to M.~Kerz. In \S\ref{ss:compactness} we give a variant of his proof, in which Lemma~\ref{l:Faltings} is replaced with a compactness argument. 

Recall that the set of closed points of $X$ is denoted by $|X|$.
For each $x\in |X|$ we have the geometric Frobenius $F_x$, which is a conjugacy class in
$\pi_1 (X)$.

\subsection{Constructing a representation of $\pi_1(X)$.} \label{ss:constructing}
Let $H\subset \pi_1 (X)$ be a closed normal subgroup satisfying properties (a)-(b) from
Definition~\ref{d:LS'} with respect to our $f\in\LS'_r(X)$. By Proposition~\ref{p:Hilbert-pro-l}, there exists an irreducible regular arithmetic curve $C$ with a morphism $\varphi :C\to X$ such that the map $\varphi_*:\pi_1(C)\to\pi_1 (X)/H$ is surjective. Take any such pair $(C,\varphi )$. Then
$\varphi^* (f)$ comes from a semisimple representation $\rho_C:\pi_1(C)\to GL(r,E_{\lambda})$.
After an appropriate conjugation, $\rho_C$ becomes a homomorphism 
$\pi_1(C)\to GL(r,O)$. 

\begin{lem}   \label{l:Kerz1}
$\Ker\rho_C \supset H_C$, where $H_C:=\Ker (\varphi_*:\pi_1(C)\to\pi_1 (X)/H)$.
\end{lem}

\begin{proof}
Since $\rho_C$ is semisimple and $H_C$ is normal the restriction of $\rho_C$ to $H_C$
is semismple. So it remains to show that $\rho_C (h)$ is unipotent for all $h\in H_C$.

Property (b) from Definition~\ref{d:LS'} implies that for every nonzero ideal $I\subset O$ there exists an open normal subgroup $U_I\subset\pi_1 (C)$ such that
for every $c\in |C|$ with $F_c\in U_I$ the polynomial $\det (1-t\rho_C (F_c))$ is congruent to
$(1-t)^r$ modulo $I$. So by \v {C}ebotarev density and continuity of $\rho_C$, 
\begin{equation}   \label{e:Kerz1}
\det (1-t\rho_C (h))\equiv (1-t)^r\mbox{ mod } 
I \,\mbox{ for all } h\in U_I \,.
\end{equation}
But $H_C$ is contained in each of the $U_I$'s, so \eqref{e:Kerz1} implies that
$\rho_C (h)$ is unipotent for all $h\in H_C$.
\end{proof}

By Lemma~\ref{l:Kerz1}, we can consider $\rho_C$ as a homomorphism 
$\pi_1 (X)/H\to GL(r,O)$. By construction, the equality
\begin{equation}   \label{e:Kerz2}
\det (1-t\rho_C (F_x))=f_x(t)
\end{equation}
holds if $x\in\varphi (|C|)$ and $\varphi^{-1}(x)$ contains a point whose residue field is equal to
that of $x$. To prove Proposition~\ref{p:b}, we will now show that \eqref{e:Kerz2} holds for 
\emph{all} $x\in |X|$.

\subsection{Using a lemma of Faltings} \label{ss:Faltings}
\begin{rem}   \label{r:Kerz}
The proof of Proposition~\ref{p:Hilbert-pro-l} shows that the group 
$\pi_1 (X)/H=\pi_1(C)/H_C$ is topologically finitely generated.
\end{rem}

\begin{lem}  \label{l:Faltings}
\begin{enumerate}
\item[(i)] 
There exists a finite subset $T\subset |C|$ such that any semisimple representations
$\rho_1,\rho_2 :\pi_1(C)/H_C\to GL(r,E_{\lambda})$ with 
\[
\Tr\rho_1 (F_c)=\Tr\rho_2 (F_c) \;\mbox{ for all } c\in T
\] 
are isomorphic.

\item[(ii)] 
If $C$ is flat over $\BZ$ this is true for $\pi_1(C)$ instead of $\pi_1(C)/H_C\,$.
\end{enumerate}
\end{lem}

\begin{proof}
Statement (ii) is due to Faltings (see \cite[Satz 5]{Fa} or \cite[Theorem 3.1]{DeF}). The proof of (ii) 
uses only the finiteness of the set of homomorphisms from $\pi_1(C)$ to any fixed finite group
(which holds if $C$ is flat over $\BZ$). So by Remark~\ref{r:Kerz}, the same argument proves (i).
\end{proof}

Let $T$ be as in Lemma~\ref{l:Faltings}(i). We have to show that \eqref{e:Kerz2} holds for any 
$x\in |X|$. By Proposition~\ref{p:Hilbert-pro-l}, there exists an irreducible regular arithmetic curve $C'$ with a morphism $\varphi' :C'\to X$ such that the map
$\varphi'_*:\pi_1(C')\to\pi_1 (X)/H$ is surjective and for each $y\in T\cup\{ x\}$ there exists
a point $z\in (\varphi')^{-1}(y)$ whose residue field is equal to that of $y$. Applying the argument of
\S\ref{ss:constructing} to $(C',\varphi')$ we get a semisimple representation
$\rho_{C'}:\pi_1 (X)/H\to GL(r,E_{\lambda})$ such that $\det (1-t\rho_C (F_y))=f_y(t)$
for each $y\in T\cup\{ x\}$. By the above choice of $T$, this implies that the representations 
$\pi_1(C)\to GL(r,E_{\lambda})$ corresponding to $\rho_C$ and $\rho_{C'}$ are isomorphic.
So $\rho_C\simeq\rho_{C'}$ and therefore $\det (1-t\rho_C (F_x))=f_x(t)$, QED.

\subsection{Using a compactness argument} \label{ss:compactness}
Instead of referring to the  lemma of Faltings, one can finish the proof of
Proposition~\ref{p:b} as follows. We have to prove the existence of a homomorphism
$\rho :\pi_1 (X)\to GL(r,O)$ such that 
\begin{equation}   \label{e:Kerz3}
\det (1-t\rho (F_x))=f_x(t)
\end{equation}
for all $x\in |X|$. By Remark~\ref{r:Kerz}, the group $\pi_1 (X)/H$ is topologically finitely generated.
So the set
\[
Z:=\Hom (\pi_1 (X)/H, GL(r,O))= \leftlim_{I\ne 0} \Hom (\pi_1 (X)/H, GL(r,O/I))
\]
equipped with the topology of projective limit is compact.

For each $x\in |X|$ let $Z_x\subset Z$ denote the set of all homomorphisms
$\rho: \pi_1 (X)/H\to GL(r,O)$ satisfying \eqref{e:Kerz3}.
Then each $Z_x$ is a closed subset of 
$Z$. We have to show that the intersection of these subsets is non-empty.
But $Z$ is compact, so it suffices to prove that 
for any finite subset $S\subset |X|$ the set
\begin{equation}   \label{e:Kerz4}
\bigcap_{x\in S}Z_x 
\end{equation}
is non-empty. By Proposition~\ref{p:Hilbert-pro-l}, there exists an irreducible regular arithmetic curve $C$ with a morphism $\varphi :C\to X$ and a scheme-theoretical section $S\to X$ such that the map $\varphi_*:\pi_1(C)\to\pi_1 (X)/H$ is surjective. The set \eqref{e:Kerz4} contains the 
homomorphism $\rho_C: \pi_1 (X)/H\to GL(r,O)$ constructed in \S\ref{ss:constructing}, so it is 
non-empty.
%
%
%

\section{Proof of Proposition~\ref{p:c}} \label{s:proofc} 

\subsection{A specialization lemma}   \label{ss:specialization_lemma} 
We will need the following elementary lemma. (For deeper statements in this spirit, see \cite[\S2.4]{W1}, \cite[Lemma 2.1 and Proposition 2.3]{KS1}, \cite[Lemma 2.4]{KS2}.)

\begin{lem}       \label{l:specialization} 
Let $X$ be a regular scheme of finite type over $\BZ$, $D\subset X$ an irreducible divisor,
$G$ a finite group, and $\pi :Y\to X\setminus D$ a $G$-torsor ramified at $D$. Then there
exists a closed point $x\in D$ and a 1-dimensional subspace $l$ of the tangent space $T_xX:=(\fm_x/\fm_x^2)^*$
with the following property: 
\begin{enumerate}
\item[$(\bigstar)$]
if $C\subset X_{\bar x}$ is any regular 1-dimensional closed subscheme tangent to $l$ such that
$C\not\subset D_{\bar x}$ then the pullback of $\pi :Y\to X\setminus D$ to $C\setminus \{ \bar x\}$ 
is ramified at $\bar x$.
\end{enumerate}
Here $\bar x$ is a a geometric point corresponding to $x$ and $X_{\bar x}$, $D_{\bar x}$ are the strict Henselizations.
\end{lem}

\begin{proof}
We will show that {\em at least one} of the following statements (a) and (b) is true:
\begin{enumerate}
\item[(a)] property $(\bigstar)$ holds for any closed point $x$ of some non-empty open
subset $U\subset D$ and for any $l\not\subset T_xD$; 

\item[(b)] $D$ is a variety over a finite field and there exists a conic closed subset $F$ of the
tangent bundle $TD$ such that $F\ne TD$ and property $(\bigstar)$ holds whenever
$l\subset T_xD$ and $l\not\subset F$.
\end{enumerate}
Let $I\subset G$ be the inertia subgroup at the generic point of $D$.
To prove that (a) or (b) holds, we can replace $X$ by any scheme $X'$ etale over $X$ with 
$D\times_XX'\ne\emptyset$ and replace $Y$ by $Y\times_XX'$. So we can assume that 
$I=G$. Then $G$ is solvable (because $I$ is). So we can
assume that $|G|$ is a prime number $p$ (otherwise replace $Y$ by $Y/H$, where $H\subset G$
is a normal subgroup of prime index).

Let $\bar Y$ be the normalization of $X$ in the ring of fractions of $Y$. We have a finite morphism
$\bar\pi :\bar Y\to X$ and an action of $G$ on $\bar Y$ such that $\bar Y/G=X$. 
After shrinking $X$ we can assume that $\bar Y$ is regular.

Set $\tilde D:=(\bar\pi^{-1}(D))_{\red}$. The assumption $I=G$ means that the action of $G$ on
$\tilde D$ is trivial and the morphism $\bar\pi_D:\tilde D\to D$ is purely inseparable. 
Let $e_1$ be its degree and let $e_2$ be the multiplicity of $\tilde D$ in the divisor $\bar\pi^{-1}(D)$. Then $e_1e_2=|G|=p$, so $e_1$ equals 1 or $p$.

Case 1: $e_1=1$, $e_2=p$. Let us show that (a) holds for $U=D$. Since $e_1=1$ the 
morphism $\bar\pi_D:\tilde D\to D$ is an isomorphism. So if $l$ is as in (a) and
$C\subset X_{\bar x}$ is any regular 1-dimensional closed subscheme tangent to $l$
then $T_{\bar x}C$ is transversal to the image of the tangent map $T_{\bar\pi^{-1}(x)}\bar Y\to T_xX$. Therefore $C\times_X\bar Y$ is regular. On the other hand, the fiber of 
$\bar\pi$ over $\bar x$ has a single point. So the pullback of 
$\pi :Y\to X\setminus D\,$ to $C\setminus \{ \bar x\}$ is indeed ramified at $\bar x$.

Case 2: $e_1=p$, $e_2=1$. Then $D$ and $\tilde D$ are varieties over $\BF_p$.
After shrinking $X$ we can assume that $D$ and $\tilde D$ are smooth and the differential
of the morphism $\bar\pi_D: \tilde D\to D$ has constant corank~$\delta$. Since $e_1=p$ one has 
$\delta =1$. The images of the differential of $\bar\pi_D$ at various points of $\tilde D$ form a vector subbundle of $(TD)\times_D\tilde D$ of codimension 1. Let $F\subset TD$ be its image; then $F$ is  a conic closed subset. We will show that (b) holds for this $F$. Let $x\in D$ be a closed point and 
$H_x:=\im \, (T_{\bar\pi^{-1}(x)}\bar Y\to T_xX)$. Since $\delta =1$ and $e_2=1$ one has 
$\codim H_x=1$ . So if $l$ is as in (b) then $l$ is transversal to $H_x$. Now one can finish the argument just as in Case 1.
\end{proof}

\begin{cor}       \label{c:specialization} 
Let $X$ be a regular scheme of finite type over $\BZ [\ell^{-1}]$ and
$U\subset X$ a dense open subset. Let $\E $ a lisse $E_{\lambda}$-sheaf on $U$. 
Suppose that $\E $ does not extend to a lisse $E_{\lambda}$-sheaf on $X$. 
Then there exists a closed point $x\in X\setminus U$ and a 1-dimensional subspace 
$l\subset T_xX$
with the following property: 
\begin{enumerate}
\item[$(*)$]
if $C$ is a regular arithmetic curve, $c\in C$ a closed point, and 
$\varphi :(C, c )\to (X, x)$ a morphism 
with $\varphi^{-1}(U)\ne\emptyset$ such that  the image of the tangent map 
$T_ cC\to T_xX\otimes_{k_x}k_c$ equals $l\otimes_{k_x}k_c$
then the pullback of $\E$ to $\varphi^{-1}(U)$ is ramified at $c$.
\end{enumerate}
(Here $k_x$ and $k_c$ are the residue fields.)
\end{cor}

\begin{proof}
By the Zariski-Nagata purity theorem 
\cite[expos\'e~~X, Theorem 3.4]{SGA2}, $\E$ is ramified along some irreducible
divisor $D\subset X$, $D\cap U=\emptyset$. Now use Lemma~\ref{l:specialization}.
\end{proof}

\subsection{Proof of Proposition~\ref{p:c}} \label{ss:proofc} 
Proposition~\ref{p:c} is equivalent to the following lemma.

\begin{lem}  \label{l:c}
Let $X$ be be a regular scheme of finite type over $\BZ [\ell^{-1}]$. Let $U\subset X$ be a 
dense open subset, $\E_U$ a semisimple lisse $E_{\lambda}$-sheaf on $U$, and 
$f_U\in\LS_r(U)\subset\widetilde{\LS}_r(U)$ the class of $\E_U$. Let
$f\in\widetilde{\LS}_r(X)$ be such that  $f|_U=f_U$
and $\varphi^*(f)\in\LS_r(C)$ for every regular arithmetic curve $C$ and every morphism 
$\varphi :C\to X$. Then
 \begin{enumerate}
\item[(i)] $\E_U$ extends  to a lisse $E_{\lambda}$-sheaf $\E$ on $X$;
\item[(ii)] the class of $\E$ in $\LS_r(X)\subset\widetilde{\LS}_r(X)$ equals $f$.
\end{enumerate}
\end{lem}

\begin{proof}
(i) Assume the contrary. We can assume that $X$ is irreducible. Choose a closed point 
$x\in X\setminus U$ and a 1-dimensional subspace $l\subset T_{\bar x}X$ satisfying property $(*)$ from Corollary~\ref{c:specialization}.
Let $H\subset\pi_1 (U)$ be the kernel of the representation $\rho :\pi_1 (U)\to GL(r,E_{\lambda})$ 
corresponding to $\E$. The group $\pi_1 (U)/H\simeq\im\rho$ has an open pro-$\ell$-subgroup, so by 
Proposition~\ref{p:Hilbert-pro-l}, there is an irreducible regular arithmetic curve $C$
with a closed point $c\in C$ whose residue field is isomorphic to that of $x$  and a morphism 
$\varphi :(C,c)\to (X,x)$ such that $\varphi^{-1}(U)\ne\emptyset$,
the homomorphism $\varphi_*:\pi_1(\varphi^{-1}(U))\to\pi_1 (U)/H$ is surjective,
and the image of the tangent map $T_cC\to T_xX$ equals $l$. 
The surjectivity of $\varphi_*:\pi_1(\varphi^{-1}(U))\to\pi_1 (U)/H$ implies that the pullback of
$\E$ to $\varphi^{-1}(U)$ is semisimple. So the assumptions on $f$ ensure
that the pullback of $\E$ to $\varphi^{-1}(U)$ has no ramification at $c$. This contradicts
property $(*)$ from Corollary~\ref{c:specialization}.

(ii) Let $f^{?}\in\LS (X)\subset\widetilde{\LS}_r(X)$ be the class of $\E$. We have to show that 
$f^{?}(x)=f(x)$ for every closed point $x\in X$. 
Choose a triple  $(C,c,\varphi )$, where $C$ is a
regular arithmetic curve,  $c\in C$  is a closed point whose residue field is isomorphic to 
that of $x$ and $\varphi :(C,c)\to (X,x)$ is a  morphism such that $\varphi^{-1}(U)\ne\emptyset$.
Then $\varphi^*f^{?},\varphi^*f\in\LS_r(C)$ have equal images in $\LS_r(\varphi^{-1}(U))$.
So $\varphi^*f=\varphi^*f$ and therefore $f^{?}(x)=f(x)$.
\end{proof}


\begin{rem}   \label{r:geom_semisimplicity}
If $X$ is over $\BF_p$ then in the proof of Lemma~\ref{l:c}(i) the surjectivity of
$\varphi_*:\pi_1(\varphi^{-1}(U))\to\pi_1 (U)/H$ is not essential. Indeed, for our purpose it
suffices to know that the pullback of $\E$ to $\varphi^{-1}(U)$ is \emph{geometrically} semisimple,
and this follows from \cite[Theorem 3.4.1 (iii)]{De} combined with \cite[Remark 1.3.6]{De} and
\cite[Proposition VII.7 (i))]{La}.
\end{rem}

\section{Counterexamples}   \label{s:counterex}
In this section we 
give two examples showing
that in Theorem~\ref{Mainth} the regularity assumption on $X$ cannot be replaced by normality.
In both of them one has a normal scheme  $X$ of finite type over $\BZ [\ell^{-1}]$ with a unique singular point $x_0\in X$ and a desingularization $\pi: \hat X\to X$  inducing an isomorphsim
$\pi^{-1} (X\setminus  \{ x_0\})\iso X\setminus  \{ x_0\}$. One also has  an element $f\in\widetilde{\LS}_r(X)$
such that $\pi^*f\in  \LS_r(\hat X)\subset \widetilde{\LS}_r(\hat X)$ but $f\notin\LS_r(X)$.
In the first example $f\notin\LS_r(X)$ because the semisimple lisse $E_{\lambda}$-sheaf  $\E$
on $\hat X$ corresponding to $\pi^*f$ does not descend to $X$. In the second example $\E$
is constant and therefore descends to $X$, but $f(x_0)$ is ``wrong". In both examples the key
point is that $(\pi^{-1} (x_0))_{\red}$ is not normal.

\subsection{First example}  \label{ss:ex1}
\subsubsection{The idea}   \label{ss:theidea}
Let $X$, $x_0$, and $\pi: \hat X\to X$ be as above.
Let $C:=(\pi^{-1} (x_0))_{\red}$.
Let $i:C\mono \hat X$ be the embedding. 

Now let $\E$ be an $E_{\lambda}$-sheaf  on $\hat X$ of rank $r$.
It defines an element $f_{\E}\in \widetilde{\LS}_r(\hat X)$. Suppose that
 \begin{enumerate}
\item[(a)] $\E$ is semisimple;
\item[(b)] $i^*\E$ is not geometrically constant;
\item[(c)] for every $c\in |C|$ the polynomial $f(c)\in P_r (O)$ equals the ``trivial" polynomial
$(1-t)^r\in P_r(O)$.
\end{enumerate}

Define $f\in\widetilde{\LS}_r(X)$ as follows:
$f|_{X\setminus  \{ x_0\}}:=f_{\E}|_{X\setminus  \{ x_0\}}$ and $f(x_0)$ is the 
polynomial $(1-t)^r\in P_r(O)$. Then it is easy to see that $f$ satisfies the conditions of 
Theorem~\ref{Mainth} but $f\notin\LS_r(X)$.

It remains to construct $X$, $\hat X$, $\pi$, $\E$ with the above properties.
Note that $C$ \emph{cannot be normal:} otherwise (c) would imply (b) by \v {C}ebotarev density.

\subsubsection{A construction of $X$, $\hat X$, $\pi$, $\E$ for $r=1$} \label{ss:Sasha}
The following construction was communicated to me by A.~Beilinson.

Let $n\in\BN$. Over $\BF_q$ consider the curve $\BP^1\times (\BZ/n\BZ )$ (i.e., a disjoint union of $n$ copies of $\BP^1$).
Gluing $(\infty ,i)\in\BP^1\times (\BZ/n\BZ )$ with $(0 ,i+1)\in\BP^1\times (\BZ/n\BZ )$ for all
$i\in\BZ/n\BZ$ one gets a curve $C_n$ equipped with a free action of $\BZ/n\BZ$. At least, if
$n=3$ or $n=4$ it is easy to embed $C_n$ into a smooth quasiprojective 
surface $\hat Y $ over $\BF_q$ so that the action of $\BZ/n\BZ$ on $C_n$ extends to $\hat Y $ (take $\hat Y =\BP^2$ if $n=3$ and $\hat Y =\BP^1\times\BP^1$ if $n=4$). We can assume that the action of 
$\BZ/n\BZ$ on $\hat Y $ is free (otherwise replace $\hat Y $ by
an open subset). We can also assume that the intersection matrix of the curve $C_n\subset \hat Y $
is negative definite (otherwise pick a sufficiently big $(\BZ/n\BZ)$-stable finite reduced subscheme
$S$ of the nonsingular part of $C_n$ and replace $\hat Y $ by its blow-up at $S$). Blowing down
$C_n\subset \hat Y $ we get an algebraic surface\footnote{Since our ground field is finite,
this surface is a quasiprojective scheme rather than merely an algebraic space, see 
\cite[Theorem 4.6]{Ar}.} $Y$ with a unique singular point~$y_0$. 

Now let $\hat X$ and $X$ be the quotients of
$\hat Y$ and $Y$ by the action of $\BZ/n\BZ$. The morphism $\hat Y\to Y$ induces a birational morphism $\pi: \hat X\to X$.  Since the action of $\BZ/n\BZ$ on $\hat Y $ is free the surface 
$\hat X$ is smooth.
Finally, let $\E$ be the rank 1 local system on $\hat X$ corresponding to the $(\BZ/n\BZ)$-torsor 
$\hat Y\to\hat X$ and a nontrivial character $\BZ/n\BZ\to \dbQbar_{\ell}^{\times}$.



\subsection{Second example} \label{ss:related}
Let $\hat Y$, $Y$ and $y_0$ be as in \S\ref{ss:Sasha}. Let $\alpha\in H^1(\BF_q ,\BZ/n\BZ )$,
$\alpha\ne 0$. The group $\BZ/n\BZ$ acts on $(\hat Y ,Y,y_0)$.
Now let  $(\hat X ,X,x_0)$ denote the form of $(\hat Y ,Y,y_0)$ corresponding to $\alpha$. Then

 \begin{enumerate}
\item[(a)] $X$ is a normal surface over $\BF_q$, $x_0\in X(\BF_q )$ is a singular point,
and $\hat X$ is a desingularization of $X$;
\item[(b)] the residue field of any closed point of the preimage of $x_0$ in $\hat X$ contains
$\BF_{q^m}$, where $m>1$ is the order of~$\alpha$.
\end{enumerate}

Now we will define an element $f\in\widetilde{\LS}_1(X)$, i.e.,
a function $f:|X|\to O^{\times}$,
where $O\subset E_{\lambda}$ is the ring of integers. Namely, set 
\[
f(x):=1 \mbox{\, if \,} x\ne x_0, \quad f(x_0):=\zeta,
\]
where $\zeta\in O$, $\zeta^m=1$, $\zeta\ne 1$ (we assume that $E_{\lambda}$ is
big enough so that $\zeta$ exists). Properties (a)-(b) above imply that $f$ satisfies the conditions of Theorem~\ref{Mainth}
but $f\notin\LS_1(X)$.

\begin{rem}  \label{r:regularity_ess}
Property (b) implies that there are no
nonconstant morphisms $(C,c)\to (X,x_0)$ with $C$ being a smooth curve over $\BF_q$
and $c\in C(\BF_q)$. So Theorem~\ref{t:Hilbert} does not hold for our surface $X$ and $S=\{ x_0\}$,
and Lemma~\ref{l:regular} does not hold for the local ring of $X$ at $x_0$.
Therefore in Theorem~\ref{t:Hilbert} and Lemma~\ref{l:regular} the regularity assumption 
cannot be replaced by normality.
\end{rem}


\appendix

\section{Proof of Theorem~\ref{t:Hilbert} } \label{s:Hilbert}
We follow the proof of \cite[Lemmas 3.2-3.3]{Bl} and \cite[Lemma 20]{W1}  (with slight modifications).

\medskip

\subsection{A conventional formulation of Hilbert irreducibility}  \label{ss:conventional}
Let $k$ be a global field, i.e., a finite extension of either $\BQ$ or $\BF_p (t)$.
The following version of the Hilbert irreducibility theorem is proved in 
\cite[Ch.~9, Theorem~4.2]{Lang2} and in  \cite[Ch.~13, Proposition~13.4.1]{FJ}.

\begin{theorem}   \label{t:Hilb}
Let $U$ be an open subscheme of the affine space $\BA_k^n$. Let $\tilde U$ be an irreducible
scheme finite and etale over $U$. Let $U(k)_{\Hilb}\subset U(k)$ denote the set of those points $\alpha\in U(k)$ that are inert in $\tilde U$ (i.e., the fiber of the map $\tilde U\to U$ over $\alpha$ is the spectrum of a field). Then $U(k)_{\Hilb}\ne\emptyset$. Moreover,
for every finitely generated subring $R\subset k$ with field of fractions $k$ one has
$U(k)_{\Hilb}\cap R^n\ne\emptyset$. \hfill\qedsymbol
\end{theorem}

%

Let $T$ be a finite set of nonarchmidean places of $k$.
Set $$\hat k:=\prod\limits_{v\in T}k_v,$$ where $k_v$ is the completion of $k$.
The topology on $\hat k$ induces a topology on the set of $\hat k$-points of any algebraic
variety over $k$. The following corollary of Theorem~\ref{t:Hilb} is standard.

\begin{cor}   \label{c:Hilb-dense}
In the situation of Theorem~\ref{t:Hilb} the image of $U(k)_{\Hilb}$ in $U(\hat k )$ is dense.
\end{cor}   

\begin{proof} 
We follow the proof of \cite[Ch.~9,  Corollary~2.5]{Lang2}.  Let $R\subset k$ be a finitely generated subring with field of fractions $k$ such that $|x|_v\le 1$ for all $x\in R$, $v\in T$.
Choose a nonzero $\pi\in R$ so that $|\pi |_v < 1$ for all $v\in T$. Then every nonempty open
subset of $\hat k^n$ contains a subset of the form
$\pi^m R^n+u$, where $m\in\BZ$, $u\in k^n$.
Theorem~\ref{t:Hilb} implies that $U(k)_{\Hilb}\cap g(R^n)\ne\emptyset$ for any
$g\in\Aut (\BA_k^n )$. Setting $g(x)=\pi^m x+u$ we see that 
$U(k)_{\Hilb}\cap (\pi^m R^n+u)\ne\emptyset$. 
\end{proof} 

\subsection{Bloch's lemmas}  \label{ss:Bloch}
We follow  \cite[\S 3]{Bl} (with minor modfications).
\begin{lem}  \label{l:1Bloch}
Let $k$ and $T$ be as in \S\ref{ss:conventional}. Let $V$ be an irreducible smooth $k$-scheme
and $\tilde V$ an irreducible scheme finite and etale over $V$. Then for
any non-empty open subset $W\subset V(\hat k)$ there exists a finite extension 
$k'\supset k$ equipped with a $k$-morphism $k'\to \hat k$ and a point $\alpha\in V(k')$ such that 
$\alpha$ is inert in $\tilde V$ and the image of $\alpha$ in $V(\hat k)$ belongs to $W$.
\end{lem}

\begin{proof}
Choose a nonempty open subscheme $V'\subset V$ so that there exists a finite etale morphism $V'\to U$, where $U$ is an open subscheme of an affine space $\BA_k^n$.
Since $V' (\hat k)$ is dense in $V (\hat k)$ and
the map $V' (\hat k)\to U(\hat k)$ is open the image of $V' (\hat k)\cap W$ in $U(\hat k)$
is nonempty and open. So by Corollary~\ref{c:Hilb-dense}, it contains a point $\beta\in U(k)$ inert in $\tilde V$. Let $V'_{\beta}$ be the fiber of $V'\to U$ over $\beta$. Since  $\beta$ is inert, $V'_{\beta}=\Spec k'$ for some field $k'\supset k$. By construction, 
$V'_{\beta}$ has a $\hat k$-point in $W$. This $\hat k$-point yields a $k$-morphism $f:k'\to \hat k$. 
On the other hand, the embedding $V'_{\beta}\mono V$ defines a point $\alpha\in V( k')$. 
Clearly $k'$, $f$, and $\alpha$ have the desired properties.
\end{proof}

Let $k$, $T$, and $\hat k$ be as in \S\ref{ss:conventional}.
Set $\hat O:=\prod\limits_{v\in T}O_v$, where $O_v\subset k_v$ is the ring of integers. Set
\begin{equation}   \label{e:O}
O:=\hat O\times_{\hat k}k=\{x\in k \;\bigl\lvert \; |x|_v\le 1 \mbox{ for } v\in T\}.
\end{equation}
If $T\ne\emptyset$ then $O$ is a semilocal ring whose completion equals $\hat O$, and
if $T=\emptyset$ then $O=k$.

Let $A$ be an etale $\hat k$-algebra (i.e., $A$ is a product of fields each of which is a
finite separable extension of one of the fields $k_v$, $v\in T$). Let $\hat O_A\subset A$ be the integral
closure of $\hat O$ in $A$.

\begin{lem}  \label{l:2Bloch}
Let $Y$ be a scheme of finite type over $O$. Let $V\subset Y\otimes_O k$ be an irreducible open
subscheme smooth over $k$  and $\tilde V$ an irreducible scheme finite and etale over $V$.
For any field $K\supset k$ let $V(K)_{\Hilb}\subset V(K)$ denote the set of those $K$-points of
$V$ that are inert in~$\tilde V$. Let $\alpha_0\in Y(\hat O_A)\times_{Y(A)}V(A)\subset Y(\hat O_A)$
and let $I\subset\hat O_A$ be an open ideal.
 Then
there exists a finite extension 
$k'\supset k$ equipped with a $k$-morphism $k'\to A$ and an element 
\[
\alpha\in Y(O')\times_{Y(k')}V(k')_{\Hilb}, \quad O':=k'\times_A\hat O_A 
\]
such that the image of $\alpha$ in  $Y(\hat O_A)$ 
is congruent to $\alpha_0$ modulo $I$.
\end{lem}


\begin{proof}
By weak approximation, there exists a finite extension $\tilde k\supset k$ and a finite set $\tilde T$
of nonarchimedian places of $\tilde k$ such that $A=\prod\limits_{w\in\tilde T}\tilde k_w$.
Setting $\tilde O:= \{x\in\tilde k \;\bigl\lvert \; |x|_w\le 1 \mbox{ for } w\in\tilde T\}$ and replacing
$Y$ and $V$ with $Y\otimes_O \tilde O$ and $V\otimes_k \tilde k$ we reduce the lemma to
the particular case where $A=\hat k$ (and therefore $\hat O_A=\hat O$).

In this case let $W$ be the set of all elements of $Y(\hat O)$ congruent to $\alpha_0$ modulo $I$.
Let $W_1\subset Y(\hat k )$ be the image of $W$ and $W_2:=W_1\cap V(\hat k )$.
The subsets $W_1, W_2\subset Y(\hat k )$ are open and nonempty.
Applying Lemma~\ref{l:1Bloch} to $V$, $\tilde V$, and $W_2\subset V(\hat k)$ 
we get a finite extension $k'\supset k$ equipped with a $k$-morphism $k'\to \hat k$ and an element 
$\alpha\in Y(\hat O)\times_{Y(\hat k)}V(k')_{\Hilb}$
such that the image of $\alpha$ in  
$Y(\hat O)$ belongs to $W$.
It remains to note that if one sets $\quad O':=k'\times_{\hat k}\hat O$ then
$$Y(\hat O)\times_{Y(\hat k)}V(k')_{\Hilb}=Y(O')\times_{Y(k')}V(k')_{\Hilb}$$ because the
square
\[
\xymatrix{
 \Spec \hat k\ar[r]\ar[d] &\Spec \hat O\ar[d] \\
 \Spec k' \ar[r] & \Spec O'
}
\]
is co-Cartesian in the category of all schemes.
\end{proof}

\subsection{On regular local rings}   \label{ss:regular}
In this subsection we prove Lemma~\ref{l:regular}, which will be used in \S\ref{ss:proofofth}.
\begin{lem}      \label{l:prime_ideals}
Let $R$ be a ring and $J\subset R$ an ideal. Let $\fp_1$, \ldots, $\fp_n$ be prime ideals not
containing $J$. Then for every $r_0\in R$ there exists $r\in r_0+J$ such that $r\notin\fp_i$ for each $i$.
\end{lem}

\begin{proof}     
We can assume that $\fp_i\not\subset\fp_j$ for $i\ne j$. By induction, we can also assume that 
$r\notin\fp_i$ for $i<n$. By assumption, $J\fp_1\ldots\fp_{n-1}\not\subset\fp_n$.
Let $u\in J\fp_1\ldots\fp_{n-1}$, $u\not\in\fp_n$. Then set $r:=r_0+u$ if $r_0\in\fp_n$ and
$r:=r_0$ if $r_0\not\in\fp_n$.
\end{proof}

\begin{lem}      \label{l:regular}
Let $R$ be a regular local ring with maximal ideal $\fm$. Let $l\subset (\fm /\fm^2)^*$ be a 1-dimensional subspace and $D\subset\Spec R$ a divisor. Then there is a regular
1-dimensional closed subscheme $C\subset\Spec R$ tangent to $l$ such that $C\not\subset D$.
\end{lem}

\begin{proof}     
It suffices to show that if $\dim R>1$ then there is a regular closed subscheme $Y\subset\Spec R$
such that $Y\not\subset D$ and the tangent space of $Y$ at its closed point 
contains $l$ (then one can replace $\Spec R$ with $Y$ and proceed by induction).

To construct $Y$, choose $r_0\in\fm$ so that the image of $r_0$ in $\fm /\fm^2$ is nonzero and
orthogonal to $l$. By Lemma~\ref{l:prime_ideals}, there exists $r\in r_0+\fm^2$ such that $r$
does not vanish on any irreducible component of $D$. Now set $Y:=\Spec R/(r)$.
\end{proof}

Note that in Lemma~\ref{l:regular} and Theorem~\ref{t:Hilbert} the regularity assumption 
cannot be replaced by normality, see Remark~\ref{r:regularity_ess} at the end of \S\ref{s:counterex}. 

\subsection{Proof of Theorem~\ref{t:Hilbert}}
\label{ss:proofofth}
 Let $\tilde U$ be the covering of $U$ corresponding to $H\subset\pi_1(U)$.
Then the surjectivity condition in Theorem~\ref{t:Hilbert}(i) means that $C\times_X\tilde U$ is
irreducible. Let us consider two cases.

\subsubsection{The case where $X\otimes\BQ\ne\emptyset$} \label{sss:unequal}
If $S=\emptyset$ then applying Lemma~\ref{l:2Bloch} for $k=\BQ$, $T=\emptyset$, and
$Y=V=U\otimes\BQ$ one gets a finite extension $k'\supset\BQ$ and a point $\alpha\in U(k')$
inert in $\tilde U$. Then it remains to choose $C$ and $\varphi : C\to X$ so that
the generic point of $C$ equals $\Spec k'$ and the restriction of
$\varphi$ to $\Spec k'$ equals $\alpha$.

If $S\ne\emptyset$ then apply Lemma~\ref{l:2Bloch} 
to the following $k$, $T$, $Y$, $V$, $A$, $\alpha_0$, and $I$.
As before, set $k=\BQ$. Define $T$ to be the image of $S$ in $\Spec\BZ$;
then the ring $O$ defined by \eqref{e:O} is the ring of rational numbers whose denominators
do not contain primes from $T$. Set $Y:=X\otimes O$, $V:=U\otimes\BQ$. 

To define $A$, $\alpha_0$, and $I$, proceed as follows.
Let $\hat O_{X,s}$ denote the completed local ring of $X$
at $s\in S$. Lemma~\ref{l:regular} and the regularity assumption\footnote{This place (and a similar one in \S\ref{sss:char=p}) is the only part of the proof of Theorem~\ref{t:Hilbert} 
where we use that $X$ is regular.}
on $X$ allow to choose for each $s\in S$ 
a 1-dimensional regular closed subscheme
$\hat C_s\subset\Spec\hat O_{X,s}$ tangent to $l_s$ such that $\hat C_s\times_XU\ne\emptyset$
and $\hat C_s\otimes\BQ\ne\emptyset$. Let  $\hat O_A$ be the ring of 
regular functions on $\coprod\limits_{s\in S}\hat C_s $. Let $A$ be the ring of fractions of $\hat O_A$.
The morphism 
\begin{equation}    \label{e:curves1}
\Spec \hat O_A=\coprod\limits_{s\in S}\hat C_s\to X
\end{equation}
 defines an element $\alpha_0\in Y(\hat O_A)$. Let $I\subset\hat O_A$ be the square of the 
 Jacobson radical of $\hat O_A$ .

 Now let $k'$, $O'$, and $\alpha$ be as in Lemma~\ref{l:2Bloch}. Let $\tilde C$ be the spectrum of 
 the integral closure of $\BZ$ in $k'$. Then there exists an open subscheme $C\subset\tilde C$
 containing $\Spec O'$ such that $\alpha :\Spec O'\to X$ extends to a morphism
 $\varphi : C\to X$. The pair $(C,\varphi )$ has the desired properties. 

\subsubsection{The case where $X$ is over $\BF_p$}  \label{sss:char=p}
After shrinking $U$ we can assume that there is a smooth morphsim $t:U\to\BP^1$, where 
$\BP^1:=\BP^1_{\BF_p}$. 

Just as in \S\ref{sss:unequal}, we will apply Lemma~\ref{l:2Bloch} for certain
$k$, $T$, $Y$, $V$, $A$, $\alpha_0$, and $I$. Take $k$ to be the field of rational functions on
$\BP^1$ and let $V$ be the generic fiber of $t:U\to\BP^1$. Choose formal curves $\hat C_s$
as in \S\ref{sss:unequal} but instead of requiring $\hat C_s\otimes\BQ\ne\emptyset$ require
the pullback of $dt$ to $\hat C_s$ to be nonzero. Define $A$, $\hat O_A$, $I$, and the morphism
\eqref{e:curves1} just as in \S\ref{sss:unequal}.

It remains to define $T$, $Y$, and $\alpha_0$. The rational map $t:\Spec A\to\BP^1$ extends to
a regular map 
\begin{equation}    \label{e:curves2}
\Spec \hat O_A=\coprod\limits_{s\in S}\hat C_s\to  \BP^1 .
\end{equation}
Let $T\subset\BP^1$ be the image of
the composition $S\mono\coprod\limits_{s\in S}\hat C_s\to \BP^1$. Define $O\subset k$ by
\eqref{e:O}. Set $Y:=\hat X\times_{\BP^1}\Spec O$, where $\hat X\subset X\times\BP^1$ is the closure
of the graph of $t:U\to\BP^1$. The morphisms  \eqref{e:curves1} and \eqref{e:curves2}
define a morphism $\alpha_0:\Spec \hat O_A\to Y$.

Now apply Lemma~\ref{l:2Bloch} similarly to  \S\ref{sss:unequal}. \hfill\qedsymbol

\section{Weakly motivic $\BQbar_{\ell}$-sheaves and $\BQbar_{\ell}$-complexes}
\label{s:AppendixB}

In \S\ref{sss:mot2} we defined the category of weakly motivic $\BQbar_{\ell}$-sheaves
$\Sh_{\mot} (X,\BQbar_{\ell})$ and the category of weakly motivic $\BQbar_{\ell}$-complexes
$\sD_{\mot} (X,\BQbar_{\ell})$, see Definition~\ref{d:motivic}. These are full subcategories of
the corresponding mixed categories $\Sh_{\mix} (X,\BQbar_{\ell})$ and
$\sD_{\mix}(X,\BQbar_{\ell})$, see Remark~\ref{r:mixed}(i).  In this appendix we show that the category 
$\sD_{\mot}(X,\BQbar_{\ell})$ is stable under all ``natural" functors (similarly to the well known
results about $\sD_{\mix}$).

\begin{lem}    \label{l:units}
Let $f:X\to Y$ be a morphism between schemes of finite type over $\BF_p$.
Suppose that a $\BQbar_{\ell}$-sheaf $M$ on $X$ has the following property:
the eigenvalues of the geometric Frobenius acting on each stalk of $M$ are algebraic numbers
which are units outside of $p$. Then this property holds for the sheaves $R^if_!M$ and $R^if_*M$.
\end{lem}

\begin{proof}
The statement about $R^if_!M$  immediately follows from Theorems 5.2.2 and 5.4 of \cite[expos\'e XXI]{SGA7} (which were proved by Deligne in a very nice way \emph{before} his works on the 
Weil conjecture). 

To prove the statement about $R^if_*M$, use the arguments from Deligne's proof of
Theorem~5.6 of \cite[expos\'e XXI]{SGA7}. This theorem was conditional, under the assumption
that Hironaka's theorem holds over~$\BF_p$. But instead of this assumption one can use
de Jong's result on alterations \cite[Theorem~4.1]{dJ}. 
\end{proof}

\begin{rem}
The statement about $R^if_!M$ from Lemma~\ref{l:units} can also be deduced from \cite[Theo\-rem~VII.6]{La}.
\end{rem}

\begin{theorem}   \label{t:4functors}
Let $f:X\to Y$ be a morphism between schemes of finite type over $\BF_p$. Then
\begin{enumerate}
\item[(i)] the functor $f_!:\sD (X,\BQbar_{\ell})\to \sD (Y,\BQbar_{\ell})$ maps
 $\sD_{\mot}(X,\BQbar_{\ell})$ to  $\sD_{\mot}(Y,\BQbar_{\ell})$;
\item[(ii)]  the functor $f_*:\sD (X,\BQbar_{\ell})\to \sD (Y,\BQbar_{\ell})$ maps
 $\sD_{\mot}(X,\BQbar_{\ell})$ to  $\sD_{\mot}(Y,\BQbar_{\ell})$;
\item[(iii)]  the functors $f^*$ and $f^!$ map
 $\sD_{\mot}(Y,\BQbar_{\ell})$ to  $\sD_{\mot}(X,\BQbar_{\ell})$.
\end{enumerate}
\end{theorem}

\begin{proof}
(i) Combine Lemma~\ref{l:units} with Theorem 3.3.1 from \cite{De}, which says that $f_!$ maps
 $\sD_{\mix}(X,\BQbar_{\ell})$ to  $\sD_{\mix}(Y,\BQbar_{\ell})$. 
 
 (ii) Combine Lemma~\ref{l:units} with Theorem 6.1.2 from \cite{De}, which says that $f_*$ maps
 $\sD_{\mix}(X,\BQbar_{\ell})$ to  $\sD_{\mix}(Y,\BQbar_{\ell})$. Alternatively, one can deduce
 (ii) from (i) using the same argument as in Deligne's proof of the above-mentioned theorem
 (see \cite[Theorem 6.1.2]{De} or \cite[ch.~II, Theorem~9.4]{KW}).
 
%
 
 (iii) For $f^*$ the statement is obvious. For $f^!$ it follows from (ii),
 just as in the proof of \cite[Corol\-lary~1.5]{De-finitude}.
 \end{proof}

\begin{theorem}   \label{t:HOMs}
For any scheme $X$ of finite type over $\BF_p$, the full subcategory $\sD_{\mot}(X,\BQbar_{\ell})\subset\sD (X,\BQbar_{\ell})$ is  stable with respect to the functor $\otimes$, the Verdier duality functor
$\BD$, and the internal $\HOM$ functor.
\end{theorem}

\begin{proof}
The statement for $\otimes$ is obvious. 
The other two statements follow from Theorem~\ref{t:4functors}, just as in 
\cite[\S II.12]{KW} and in the proof of \cite[Corollary 1.6]{De-finitude}.
\end{proof}

\begin{defin}   \label{d:equiv}
Say that two invertible $\BQbar_{\ell}$-sheaves $A$ and $A'$ on $\Spec\BF_p$ are equivalent if
$A'A^{-1}$ is weakly motivic. Let $S$ denote the set of equivalence classes.
\end{defin}

\begin{rem}   \label{r:S}
The set $S$ equipped with the operation of tensor product is an abelian group. It is easy to see
that the abelian group $S$ is a vector space over $\BQ$.
\end{rem}

\begin{theorem}     \label{t:bigoplus}
Let $\pi:X\to\Spec\BF_p$ be a morphism of finite type.
For any invertible $\BQbar_{\ell}$-sheaf $A$ on $\Spec\BF_p$ let
$\sD_A(X,\BQbar_{\ell})\subset\sD(X,\BQbar_{\ell})$ be the essential image of 
$\sD_{\mot}(X,\BQbar_{\ell})$ under the functor of tensor multiplication by $\pi^*A$
(clearly $\sD_A(X,\BQbar_{\ell})$ depends only on the class of $A$ in $S$).
Then 
\begin{equation}   \label{e:derived-bigoplus}
\sD(X,\BQbar_{\ell})=\bigoplus_{A\in S}\sD_A(X,\BQbar_{\ell})\, .
\end{equation}
\end{theorem}

\begin{proof}
(i) Let us prove that the triangulated category 
$\sD(X,\BQbar_{\ell})$ is generated by the subcategories $\sD_A(X,\BQbar_{\ell})$.
Clearly the triangulated category $\sD(X,\BQbar_{\ell})$ is generated by objects of the from
$i_!\E$, where $i:Y\mono X$ is a locally closed embedding with $Y$ normal connected and
$\E$ is an irreducible lisse $\BQbar_{\ell}$-sheaf on $Y$. So it remains to show that for
any such $Y$ and $\E$ there exists an invertible $\BQbar_{\ell}$-sheaf $A$ on $\Spec\BF_p$
such that $\E\otimes \pi^* A^{-1}$ is weakly motivic. By \cite[\S 1.3.6]{De}, there exists 
$A$ such that the determinant of $\E\otimes \pi^* A^{-1}$ has finite order.
Since $\E\otimes \pi^* A^{-1}$ is an irreducible lisse $\BQbar_{\ell}$-sheaf whose determinant
has finite order it is weakly motivic (and pure) by a result of Lafforgue \cite[Proposition VII.7]{La}.

%

(ii) It remains to show that the subcategories $\sD_A(X,\BQbar_{\ell})$ are orthogonal to each other.
In other words, we have to prove that if $M_1,M_2\in\sD_{\mot}(X,\BQbar_{\ell})$, $A$ is
an invertible  $\BQbar_{\ell}$-sheaf on $\Spec\BF_p$ and $\Ext^i(M_1\otimes\pi_*A,M_2)\ne 0$ for
some $i$ then $A$ is weakly motivic. But 
\begin{equation}  \label{e:adjunct}
\Ext^i(M_1\otimes\pi_*A,M_2)=\Ext^i(A,\pi_*\HOM (M_1,M_2))\, ,
\end{equation}
and  $\pi_*\HOM (M_1,M_2)$ is weakly motivic by Theorems~\ref{t:HOMs}  and \ref{t:4functors}(ii). So if the r.h.s. of \eqref{e:adjunct} is nonzero then $A$ has to be weakly motivic.
\end{proof}

As before, we write $\Sh (X,\BQbar_{\ell})$ for the category of $\BQbar_{\ell}$-sheaves on $X$.
Let $\Perv (X,\BQbar_{\ell})\subset\sD(X,\BQbar_{\ell})$ denote the category of perverse 
$\BQbar_{\ell}$-sheaves.

\begin{cor}   \label{c:bigoplus}
One has 
\begin{equation}   \label{e:Sh-bigoplus}
\Sh (X,\BQbar_{\ell})=\bigoplus_{A\in S}\Sh_A(X,\BQbar_{\ell}),  
\quad\quad \Sh_A(X,\BQbar_{\ell}):=\Sh (X,\BQbar_{\ell})\cap\sD_A(X,\BQbar_{\ell}),
\end{equation}

\begin{equation} \label{e:Perv-bigoplus}
\Perv (X,\BQbar_{\ell})=\bigoplus_{A\in S}\Perv_A(X,\BQbar_{\ell}),
\quad\quad \Perv_A(X,\BQbar_{\ell}):=\Perv (X,\BQbar_{\ell})\cap\sD_A(X,\BQbar_{\ell}),
\end{equation}
where $S$ is as in Definition~\ref{d:equiv}.
\end{cor}

\begin{proof}
This follows from Theorem~\ref{t:bigoplus} because 
the subcategories $\Sh (X,\BQbar_{\ell})\subset\sD(X,\BQbar_{\ell})$ and 
$\Perv (X,\BQbar_{\ell})\subset\sD(X,\BQbar_{\ell})$ are 
closed under direct sums and direct summands.
\end{proof}

\begin{rem}   \label{r:A=0}
The category $\Perv_{\mot}(X,\BQbar_{\ell}):=\Perv (X,\BQbar_{\ell})\cap\sD_{\mot}(X,\BQbar_{\ell})$ is one of the direct summands in the decomposition \eqref{e:Perv-bigoplus} (it corresponds to the
trivial $A$). Similarly, $\Sh_{\mot}(X,\BQbar_{\ell})$ is one of the summands in 
 \eqref{e:Sh-bigoplus} and $\sD_{\mot}(X,\BQbar_{\ell})$ is one of the summands in 
 \eqref{e:derived-bigoplus}.

\end{rem}

\begin{prop}  \label{p:t-structure}
\begin{enumerate}
\item[(i)] The full subcategory $\sD_{\mot}(X,\BQbar_{\ell})\subset\sD (X,\BQbar_{\ell})$ is stable with respect to the perverse truncation functors $\tau_{\le i}$ and $\tau_{\ge i}$.
\item[(ii)] A perverse $\BQbar_{\ell}$-sheaf is weakly motivic if and only if each of its
irreducible subquotients is.
\end{enumerate}
\end{prop}

\begin{proof}
This follows from \eqref{e:derived-bigoplus}, \eqref{e:Perv-bigoplus},
and Remark~\ref{r:A=0}. On the other hand, the proposition
follows from Theorem~\ref{t:4functors}, 
just as in the proof of the corresponding statements for mixed sheaves 
\cite[5.1.6-5.1.7]{BBD}. 
\end{proof}

%


Finally, let us consider the nearby cycle functor $\Psi$.
Let $S$ be a smooth curve over $\BF_p$, $\bar s$ a geometric point of $S$
corresponding to a closed point $s\in S$, and $\Gamma:=\Gal (\bar s/s)$. Let $\eta$ be the generic point of the Henselization 
$S_s$, $\bar\eta$ a geometric point over $\eta$ of the strict Henselization $S_{\bar s}$,
 and $G:=\Gal (\bar\eta /\eta )$.
Then for any scheme $X$ of finite type over $S$
one has the nearby cycle functor 
$\Psi :\sD (X\setminus X_s,\BQbar_{\ell})\to\sD_G (X_{\bar s},\BQbar_{\ell})$,
where $\sD_G $ stands for the equivariant derived category. Once we fix a splitting of
the epimorphism $G\epi\Gamma$ we get the restriction functor 
$\sD_G (X_{\bar s},\BQbar_{\ell})\to\sD_{\Gamma} (X_{\bar s},\BQbar_{\ell})=\sD(X_s, \BQbar_{\ell})$.
It is easy to check that the preimage of $\sD_{\mot}(X_s, \BQbar_{\ell})$ in 
$\sD_G (X_{\bar s},\BQbar_{\ell})$ does not depend on the choice of the splitting. Denote it by
$\sD_{G,\mot} (X_{\bar s},\BQbar_{\ell})$. 

\begin{theorem}
The nearby cycle functor $\Psi :\sD (X\setminus X_s,\BQbar_{\ell})\to\sD_G (X_{\bar s},\BQbar_{\ell})$
maps $\sD_{\mot} (X\setminus X_s,,\BQbar_{\ell})$ to $\sD_{G,\mot} (X_{\bar s},\BQbar_{\ell})$. 
\end{theorem}

The proof below is parallel to Deligne's proof of a similar statement for $\sD_{\mix}$, see 
\cite[Theorem 6.1.13]{De}. To slightly simplify the argument, we use the theory of perverse sheaves
(which did not exist when the article \cite{De} was written).

\begin{proof}
The idea is to use the formula
\begin{equation}   \label{e:derived_invariants}
i^*j_*M=R\Gamma (I,\Psi M), \quad\quad M\in \sD (X\setminus X_s,,\BQbar_{\ell}).
\end{equation}
Here $i:X_s\to X$ and $j:X\setminus X_s\to X$ are the embeddings,
$I:=\Ker (G\epi\Gamma)$ is  the inertia group, and $R\Gamma (I,?)$ stands for the derived 
$I$-invariants.


We have to prove that if $M\in\sD (X\setminus X_s,\BQbar_{\ell})$ is weakly motivic then
so is $\Psi M$.  By the monodromy theorem, after a quasi-finite base change $S'\to S$
(which does not change $\Psi M$)
we can assume that the action of $I$ on $\Psi M$ is unipotent. Then $I$ acts via its maximal pro-$\ell$ quotient $\BZ_{\ell}(1)$.

By Proposition~\ref{p:t-structure}(i), we can also assume that $M$ is perverse.
By \cite[Corollary~4.5]{Ill}, then $\Psi M$ is perverse. So formula~\eqref{e:derived_invariants} 
implies that
\begin{equation}  \label{e:invariants}
\scrH^0 (i^*j_*M)=(\Psi M)^I,
\end{equation}
where $\scrH^0$ stands for the 0-th perverse cohomology sheaf.
By Theorem ~\ref{t:4functors} and Proposition~\ref{p:t-structure}(i), the l.h.s. of  
\eqref{e:derived_invariants} is weakly motivic. So $(\Psi M)^I$ is weakly motivic.

Let $(\Psi M)_r$ be the maximal perverse subsheaf of  $\Psi M$ on which the action of
$I$ is unipotent of degree $r$. The operators $(\sigma -1)^{r-1}$, $\sigma\in I$, define
a monomorphism $(\Psi M)_r/(\Psi M)_{r-1}\mono (\Psi M)^I(1-r) $. Since $(\Psi M)^I$ is weakly motivic,
$\Psi M$ is weakly motivic by Proposition~\ref{p:t-structure}(ii).
\end{proof}

\section{A corollary of Poonen's ``Bertini theorem over finite fields"}  \label{s:Bertini-Poonen}
In this appendix we justify Remark~\ref{r:Bertini}(ii) by proving the following proposition.

\begin{prop}   \label{p:justifying_Bertini-remark}
Let $Y\subset\BP^n_{\BF_q}$ be an absolutely irreducible closed subvariety of dimension
$d\ge 2$. Let $F'\subset F\subset Y$ be closed subschemes such that $\dim F'\le d-2$, the subscheme $Y\setminus F'$ is smooth, and $F\setminus F'$ is a smooth divisor in $Y\setminus F'$.
Then there exists an irreducible curve $C\subset Y\setminus F'$ such that
$C\not\subset F$ and for any connected etale covering $W\to Y\setminus F$ tamely ramified
along the irreducible components of $F\setminus F'$ the scheme $W\times_YC$ is connected.
Moreover, given closed points $y_i\in Y\setminus F'$ and 1-dimensional subspaces
$l_i\subset T_{y_i}Y$, $1\le i\le m$, one can choose $C$ so that $y_i\in C$ and $T_{y_i}C=l_i$
for all $i\in \{1 ,\ldots ,m\}$.
\end{prop}

The proof is given in \S\S\ref{ss:geometricBertini}-\ref{ss:Poonen} below.

\subsection{Applying a geometric Bertini theorem}  \label{ss:geometricBertini}
We will use the notation $\Pol_n$ for the set of homogeneous polynomials
in $n+1$ variables over $\BF_q\,$ (of all possible degrees). 
The scheme of zeros of $f\in\Pol_n$ in $\BP^n_{\BF_q}$ will be denoted by $V(f)$.

\begin{lem}   \label{l:geometricBertini}
Let $Y$, $F$, and $F'$ be as in Proposition~\ref{p:justifying_Bertini-remark}.
Let $C:=Y\cap V(f_1)\cap\ldots V(f_{d-1})$, where $f_1,\ldots ,f_{d-1}\in\Pol_n$. 
Suppose that $C$ is a smooth curve contained in $Y\setminus F'$ and meeting $F\setminus F'$
transversally. Then for any connected etale covering $W\to Y\setminus F$ tamely ramified
along the irreducible components of $F\setminus F'$ the scheme $W\times_YC$ is connected.
\end{lem}

\begin{proof}
Let $\overline{W}\supset W$ be the normalization of $Y$ in the field of rational functions on $W$.
We can assume that the field of constants of $W$ equals $\BF_q$ (if it equals $\BF_{q^n}$ then
consider $W$ as a covering of $Y\otimes\BF_{q^n}$). Then $W$ and $\overline{W}$ are 
absolutely irreducible. So by Theo\-rem~2.1(A) from \cite{FL} (which is a Bertini theorem),
$\overline{W}\times_YC$ is (geometrically) connected. On the other hand, the tameness
and transversality assumptions imply that $\overline{W}\times_YC$ is smooth. So any open
subset of $\overline{W}\times_YC$ is connected. In particular, $W\times_YC$ is connected.
\end{proof}

\subsection{Applying  Poonen's theorem}   \label{ss:Poonen}
Lemma~\ref{l:geometricBertini} shows that to prove Proposition~\ref{p:justifying_Bertini-remark},
it suffices to construct $f_1,\ldots ,f_{d-1}\in\Pol_n$ such that 
$Y\cap V(f_1)\cap\ldots V(f_{d-1})$ is a smooth curve contained in $Y\setminus F'$, meeting 
$F\setminus F'$ transversally, and passing through the points $y_i$ in the given directions $l_i$.
By induction, it suffices to prove the following statement.

\begin{lem}   \label{l:Poonen}
Let $Y\subset\BP^n_{\BF_q}$ be a closed subcheme of pure dimension $d\ge 2$.
Let $F'\subset F\subset Y$ be closed subschemes such that $\dim F'\le d-2$, the subscheme $Y\setminus F'$ is smooth, and $F\setminus F'$ is a smooth divisor in $Y\setminus F'$.
Then there exists $f\in\Pol_n$ such that $V(f)\cap Y$ has pure dimension $d-1$, 
$\dim (V(f)\cap F)\le d-2$, $\dim (V(f)\cap F')\le d-3$, and the schemes
$V(f)\cap (Y\setminus F')$ and $V(f)\cap (F\setminus F')$ are smooth. Moreover, 
given closed points $y_i\in Y\setminus F'$ and hyperplanes $H_i\subset T_{y_i} Y$, 
$1\le i\le m$, one can choose $f$ so that $y_i\in V(f)$ and 
$T_{y_i}(V(f)\cap Y)=H_i$ for all $i\in \{1 ,\ldots ,m\}$.
\end{lem}

\begin{proof}
Let $A$ denote the set of $f\in\Pol_n$ satisfying our conditions with a possible exception of
the condition
\begin{equation}    \label{e:possible_exception}
\dim (V(f)\cap F')\le d-3 \, .
\end{equation}
Let $B\subset\Pol_n$ be the set of $f\in\Pol_n$ such that \eqref{e:possible_exception} does not hold (i.e., such that $f$ vanishes on some irreducible component of $F'$ of dimension $d-2$).
We have to show that $A\setminus B\ne\emptyset$. In fact, $A\setminus B$ has
positive density in the sense of \cite[\S1]{Po}. If $d=2$ this directly follows from 
 \cite[Theo\-rem~1.3]{Po} (because $F'$ is finite). If $d>2$ then $B$ has density 0;
on the other hand, $A$ has positive density by \cite[Theo\-rem~1.3]{Po}.
\end{proof}

\section{Weil numbers and CM-fields}   \label{s:CM}
In this appendix we justify Remark~\ref{r:CM}.

\subsection{Formulation of the result}
Fix an algebraic closure $\BQbar\supset\BQ$. Let $R\subset\BQbar$ be the maximal totally real
subfield. Let $C\subset\BQbar$ be the union of all CM-subfields.\footnote{We include finiteness over $\BQ$ in the definition of a CM-field.} Then $C=R(\sqrt{-d})$ for any
totally positive $d\in R$.

Let $p$ be a prime and $q$ a power of $p$. A number $\alpha\in\BQbar$ is said to be a 
\emph{Weil number} if it is a unit outside of $p$ and there exists 
 $n\in\BZ$ such that all complex absolute values of $\alpha$ equal $q^{n/2}$.

\begin{theorem} \label{t:CM}
The subfield of $\BQbar$ generated by all Weil numbers equals $C$.
\end{theorem}

The fact that all Weil numbers are in $C$ is easy and well known 
(e.g., see \cite[Proposition~4]{Ho}). So Theorem~\ref{t:CM} follows from the next proposition,
which will be proved in \S\ref{ss:CMproof}.

\begin{prop}  \label{p:CM}
Any CM-field $K_0$ is contained in a CM-field $K$ which is generated by Weil numbers in $K$.
\end{prop}

\subsection{Proof of Proposition~\ref{p:CM}}  \label{ss:CMproof}
\begin{lem}  \label{l:CM}
Let $K$ be a CM-field and $k\subset K$ a totally real subfield. Suppose that there exists
a place $\fp$ of $k$ and a place $\fp'$ of $K$ such that $\fp |p$, $\fp' |\fp$,  and the map
$k_{\fp}\to K_{\fp'}$ is an isomorphism. Then Weil numbers in $K$ generate $K$ over $k$.
\end{lem}

\begin{proof}
Since the ideal class group of $K$ is finite there exists $a\in K$ such that $|a|_{\fp'}<1$ and
all other nonarchimedean absolute values of $a$ equal $1$. In particular, 
$|\bar a|_{\fp'}=|a|_{\bar\fp'}=1$ (note that the assumption $K_{\fp'}=k_{\fp}$ implies that 
$\bar\fp'\ne \fp'$). The number $\alpha:=a/\bar a$ is a Weil number. Let $\tilde K\subset K$ be the subfield generated by $\alpha$ over $k$. We will show that $\tilde K= K$.

Let $\tilde\fp$ be the place of $\tilde K$ corresponding to $\fp'$. Then $\fp'$ is the only
place of $K$ over $\tilde\fp$ (this follows from the fact that
$|\alpha |_{\fp'}<1$ and all other nonarchimedean absolute values of $\alpha$ are $\ge 1$).
On the other hand, $K_{\fp'}=\tilde K_{\tilde\fp}=k_{\fp}$.
So $[K:\tilde K]=[K_{\fp'}:\tilde K_{\tilde\fp}]=1$.
\end{proof}

\begin{proof}[Proof of Proposition~\ref{p:CM}]
Let $K_0$ be a CM-field. There exists a finite extension
$F\supset\BQ_p$ such that $K_0$ admits an embedding $i:K_0\mono F$. Fix $F$ and $i$. By weak approximation, there exists a totally real field $k$ with $k\otimes\BQ_p=\BQ_p\times F$.
Note that both $k$ and $K_0$ are subfields of $F$. The composite field 
$K_1:=k\cdot K_0\subset F$ is a CM-field. Applying Lemma~\ref{l:CM} to the place of $K_1$
corresponding to the embedding $K_1\mono F$, we see that Weil numbers in $K_1$ generate $K_1$ over $k$. 

On the other hand, let $d\in k$ be a totally positive element such that the
$\BQ_p$-algebra $k(\sqrt{-d})\otimes\BQ_p$ has $\BQ_p$ as a factor. Then by Lemma~\ref{l:CM},
Weil numbers in $k(\sqrt{-d})$ generate $k(\sqrt{-d})$ over $\BQ$. 

Now set $K:=K_1(\sqrt{-d})$. Then $K$ is a CM-field such that Weil numbers in $K$ generate $K$ over $\BQ$.
\end{proof}

\bigskip

\bibliographystyle{ams-alpha}

\end{document}